\documentclass[a4paper,11pt]{amsart}
\usepackage[margin=2cm]{geometry}
\usepackage[utf8]{inputenc}
\usepackage[english]{babel}
\usepackage{amssymb}
\usepackage{amsmath}
\usepackage{amsthm}
\usepackage{amsfonts}
\usepackage{array}
\usepackage{comment}
\usepackage{enumitem}
\usepackage{hyperref}
\usepackage{mathrsfs}
\usepackage{tikz}
\usepackage{tikz-cd, mathtools}
\usepackage{mathabx}
\usepackage{xfrac}
\usepackage{graphicx} 
\theoremstyle{plain}
\newtheorem{theorem}{Theorem}[section]
\newtheorem{lemma}[theorem]{Lemma}

\newtheorem{proposition}[theorem]{Proposition}
\newtheorem{corollary}[theorem]{Corollary}

\theoremstyle{definition}
\newtheorem{definition}[theorem]{Definition}
\newtheorem{example}[theorem]{Example}
\newtheorem{remark}[theorem]{Remark}

\newcommand{\A}{\mathbb{A}}
\newcommand{\C}{\mathbb{C}}

\newcommand{\p}{\mathbb{P}}
\newcommand{\Q}{\mathbb{Q}}
\newcommand{\R}{\mathcal{R}}
\newcommand{\Z}{\mathbb{Z}}
\newcommand{\Ox}{\mathcal{O}}

\DeclareMathOperator{\im}{im}
\DeclareMathOperator{\Pic}{Pic}

\DeclareMathOperator{\Spec}{Spec}

\DeclareMathOperator{\Proj}{Proj}

\makeatletter\let\@wraptoccontribs\wraptoccontribs\makeatother

\newcommand{\coxX}{S}
\newcommand{\pic}{\Pic}

\title{On the Cox rings of some hypersurfaces}

\begin{document}

\author{Andrew Pollock}
\address{University of Vienna, Austria}
\email{andrew.scott.pollock@univie.ac.at}

\author{Atsushi Ito}
\address{Tsukuba University, Japan}
\email{ito-atsushi@math.tsukuba.ac.jp}

\author{Bal\'{a}zs Szendr\H{o}i}
\address{University of Vienna, Austria}
\email{balazs.szendroi@univie.ac.at}

\begin{abstract}
\noindent We introduce a cohomological method to compute Cox rings of hypersurfaces in the ambient space $\p^1 \times \p^n$, which is more direct than existing methods. We prove that smooth hypersurfaces defined by regular sequences of coefficients are Mori dream spaces, generalizing a result of Ottem. We also compute Cox rings of certain specialized examples. In particular, we compute Cox rings in the well-studied family of Calabi--Yau threefolds of bidegree $(2,4)$ in $\p^1 \times \p^3$, determining explicitly how the Cox ring can jump discontinuously in a smooth family. 
\end{abstract}

\maketitle

\thispagestyle{empty}

\section*{Introduction}

Let $X$ be a smooth projective variety over~$\C$, with Picard group $\Pic X$ a free $\Z$-module for simplicity. The Cox ring 
\begin{equation*}
    \R(X) = \bigoplus_{D \in \Pic X} H^0(X,\Ox_X(D))
\end{equation*}
has the natural structure of a $\C$-algebra graded by $\Pic X$; details including a careful definition can be found in \cite{CoxBook}. Let $\iota\colon Z\hookrightarrow X$ be a smooth or mildly singular hypersurface, defined by a homogeneous equation $r\in 
\R(X)$. Under suitable conditions, the restriction map $\Pic(X)\to\Pic(Z)$
is an isomorphism; we will assume this for the rest of the Introduction, and identify the Picard groups via restriction. 
Alongside the restriction map on line bundles, we also have a restriction map between Cox rings
\[\iota^*\colon \R(X)\to \R(Z). 
\]
This map was first studied explicitly by Hausen and Artebani--Laface~\cite{hausen, AL}, giving conditions for $\iota^*$ to be surjective, leading to an isomorphism
\[\R(Z) \cong \R(X)/\langle r \rangle.\]
The aim of our paper is to give a new, explicit method to compute the Cox rings of some hypersurfaces~$Z$ in the specific ambient space $X = \p^1 \times \p^n$ for $n\geq 3$. In this case, the map $\iota^*\colon \R(X)\to \R(Z)$ is not surjective, with $O_Z(D|_Z)$ having sections that do not come from $O_X(D)$ for certain $D$.

We first define the hypersurfaces studied in this paper. Throughout, we fix $n \geq 3$ and $d,e \geq 1$.

\begin{definition}\label{def:Ztypes}
Let $Z$ be an irreducible and normal hypersurface
\begin{equation} \label{eq:generalZform}
Z = \left\{g_0x_0^d + g_1 x_0^{d-1} x_1 + \ldots + g_d x_1^d=0\right\}\subset\p^1\times\p^n
\end{equation}
defined by a set of homogeneous degree $e$ polynomials $g_0,\dots,g_d \in \C[y_0,\dots,y_n]$.
    \begin{enumerate}
        \item[(i)] We say $Z$ admits an {\em index sequence of length $l$}, with $1 \leq l \leq d$, if there are indices $0 = i_0 < \cdots < i_l = d$ such that $g_i \neq 0$ if and only if $i \in \{i_0,\dots,i_l\}$; we call the set of indices $\{i_0,\dots,i_l\}$ the {\em index sequence} of $Z$.
        \item[(ii)] We call $Z$ {\em regular with index sequence $\{i_0,\ldots, i_l\}$}, if it is non-singular, admits the index sequence $\{i_0,\dots,i_l\}$, and the homogeneous polynomials $g_{i_0},\dots,g_{i_l}$ form a regular sequence in the graded ring $\C[y_0,\dots,y_n]$.
    \end{enumerate}
\end{definition}


Note that 
amongst hypersurfaces $Z\subset\p^1\times\p^n$ admitting a given index sequence $\{i_0,\ldots, i_l\}$, being regular with index sequence $\{i_0,\ldots, i_l\}$ is an open condition. 

The following is our main result (see Theorem~\ref{thm:general}, Proposition~\ref{prop:converse}). 

\begin{theorem} \label{thm:mainintro}
Suppose that $Z\subset\p^1\times\p^n$ is a nonsingular hypersurface of the form \eqref{eq:generalZform}, 
with $n \geq 3$ and $d,e \geq 1$, admitting an index sequence $\{i_0,\dots,i_l\}$ of length $l$.
\begin{enumerate}
\item[(i)] Suppose that $Z$ is regular with index sequence $\{i_0,\dots,i_l\}$, that is $g_{i_0},\dots,g_{i_l}$ form a regular sequence in $\C[y_0,\dots,y_n]$. Then there is an isomorphism of $\Z^2$-graded algebras
\begin{equation}\label{presentationR}
    \R(Z)\cong \C[x_0,x_1,y_0,\dots,y_n,w_1,\dots,w_l]/I,
    \end{equation}
    where $\deg(x_i) = (1,0)$, $\deg (y_j) = (0,1)$, $\deg(w_k) = (i_{k-1}-i_k,e)$ for relevant $i,j,k$ and $I$ is the homogeneous ideal
    \[
    I = \langle x_1^{i_1-i_0} w_1 + g_0, x_1^{i_2-i_1} w_2 + g_1 - x_0^{i_1-i_0} w_1,\ldots, g_d - x_0^{i_l-i_{l-1}} w_l\rangle.
    \]
    In particular, $Z$ is a Mori dream space. 
\item[(ii)] Conversely, suppose that there exists an isomorphism of the form~\eqref{presentationR}, with degrees and ideal $I$ as given. Then $g_{i_0},\dots,g_{i_l}$ form a regular sequence in $\C[y_0,\dots,y_n]$. In other words, $Z$ is regular with index sequence $\{i_0,\dots,i_l\}$.
\end{enumerate}
\end{theorem}

This result substantially strengthens a result of Ottem~\cite{ottem}, who proves (i) only for general $Z$ admitting full index sequence $\{0,1,\dots,d\}$. It also provides a converse statement.

Our proofs use standard cohomological machinery to reduce the problem to an algebraic one, involving syzygies of the set $\{g_{i_0}, \dots, g_{i_l}\}$ of elements of $R$, as well as induction. Our arguments are longer than those of~\cite{ottem}, but are more direct, able to handle also the more general setting of Theorem \ref{thm:mainintro}.  More details, and further results based on the method, will be presented in~\cite{pollock}. Our examples can also be studied using work of Herrera, Laface and Ugaglia~\cite{lafaceetal}, who use yet another method involving localizations of the Cox ring $\R(X)$ of the embedding space $X=\p^1\times\p^n$; our approach gives a natural reason for the appearance of these localizations. 

For $n=3$ and $(d,e)=(2,4)$, the family 
\[{\mathcal Z} = \left\{g_0 x_0^2 + t g_1 x_0x_1 + g_2 x_1^2 =0 \right\}\subset \p^1\times \p^3\times \A^1_t
\]
is a well-studied family of Calabi--Yau threefolds, see in particular the recent~\cite{constantinetal, constantin}. In this case, Theorem~\ref{thm:mainintro} shows that this is an example of a smooth family of Calabi--Yau threefolds in which the Cox ring jumps on a closed subset of the moduli space; see Theorem~\ref{thm_cy}. We also discuss a further, singular example in this family in Example~\ref{ex:det} .

Theorem~\ref{thm:mainintro} applies in particular when $Z$ is regular and $e=1$,
with the projection $Z\to\p^1$ being a $\p^{n-1}$-bundle. If $d>n$, then the linear forms $\{g_0, \dots, g_d\}$ 
are even linearly dependent, and certainly cannot form a regular sequence, so by Theorem~\ref{thm:mainintro}(ii), the Cox ring is not of the form~\eqref{presentationR}, contrary to the statement \cite[Thm.1.1(iv)]{ottem}. We discuss one case in Example~\ref{ex:linear}-Proposition~\ref{prop:linear}.

Finally in Example~\ref{ex:regcasemodel}, we explain from our biregular point of view a birational transformation $Z\dashrightarrow Z'$ that played an important role in Ottem's arguments~\cite[Section 2]{ottem}.

Section 1 explains our cohomological approach to the problem. Section 2 is the main part of our paper, where we build up to the proof of our main results, modulo an algebraic statement that is relegated to Section 4. Section 3 discusses our examples.

\medskip

\noindent {\bf Acknowledgments and authorship } \ We thank Antonio Laface for several comments and for an argument sketched in Section 3, 
and Geoffrey Mboya for discussions about Cox rings. This paper forms part of the doctoral project of the first-named author, under the supervision of the last-named author. After an initial preprint by these two authors, the middle author made suggestions for strengthening the results. A proof of this stronger Theorem~\ref{thm:mainintro} was worked out jointly, leading to the current manuscript with an updated set of authors. 

\section{Our approach}
\subsection{Basics}

Recall our setup from the Introduction: assume that $\iota\colon Z\hookrightarrow X$ is a smooth or mildly singular hypersurface, with an isomorphism $\Pic(X)\cong\Pic(Z)$, a finitely generated free abelian group that we will denote by $\pic$. 
Let $E=[Z]\in \pic$ be the class of~$Z$. 

The Cox rings $\R(X)$ and $\R(Z)$ are both $\pic$-graded algebras
\[\R(X) = \bigoplus_{D\in\Pic} \R(X)_D, \ \ \ \R(Z) = \bigoplus_{D\in\Pic} \R(Z)_D, \]
connected by the graded homomorphism $\iota^*\colon \R(X)\rightarrow \R(Z)$. 
The interesting case for us is when the inclusion $\iota^*\R(X)\hookrightarrow \R(Z)$ is not surjective. 

The hypersurface $Z$ is given by a section $r \in \R(X)_E$. For $D\in \pic$, consider the standard exact sequence
\begin{equation}\label{seq:standard}
0\longrightarrow \Ox_X(D-E)\stackrel{\cdot r}\longrightarrow \Ox_X(D) \stackrel{\iota^*}\longrightarrow \Ox_Z(D)\longrightarrow 0. \end{equation}
The associated long exact sequence yields the short exact sequence of vector spaces
\begin{equation}\label{seq:short}
0\longrightarrow \iota^*\R(X)_D\stackrel{}\longrightarrow \R(Z)_D \stackrel{\delta}\longrightarrow N_D\longrightarrow 0,
\end{equation}
where $\iota^*\R(X)_D\cong \R(X)_D / r \cdot \R(X)_{D-E}$, and $N_D=\ker r_{D*}$ is the kernel of the induced map on first sheaf cohomology
\begin{equation}\label{seq:r*}
r_{D*} : H^1(X, \Ox_X(D-E)) \longrightarrow H^1(X, \Ox_X(D)).
\end{equation}
It will sometimes be natural to put all these maps together, to get the map
\begin{equation}
r_{*} : \bigoplus_{D\in\Pic} H^1(X, \Ox_X(D-E)) \longrightarrow \bigoplus_{D\in\Pic} H^1(X, \Ox_X(D)).
\end{equation}
We have $\iota^*\R(X)= \R(Z)$ if and only if the map $r_*$ is injective. 

\subsection{\v Cech cohomology considerations}\label{subsec:cech}

Fix a totally ordered affine cover $\mathcal{U} = \{U_i\}_{i \in I}$ of~$X$. 
From~\eqref{seq:standard}, we get the double complex

\begin{equation}\label{complex}\begin{tikzcd}
{} & 0\arrow{d} & 0\arrow{d} \\
0\arrow{r} & \displaystyle\prod_{i}\Ox_X(D-E)(U_i) \arrow{r}\arrow{d}{\cdot r} &  \displaystyle\prod_{i<j}\Ox_X(D-E)(U_i \cap U_j)\arrow{r} \arrow{d}{\cdot r}& \ldots \\
0\arrow{r} & \displaystyle\prod_{i}\Ox_X(D)(U_i)\arrow{r}{d_0} \arrow{d}{\iota^*}& \displaystyle\prod_{i<j}\Ox_X(D)(U_i \cap U_j)\arrow{r} \arrow{d}{\iota^*} & \ldots \\
0\arrow{r} &\displaystyle\prod_{i}\Ox_Z(D)(U_i)\arrow{r} \arrow{d}& \displaystyle\prod_{i<j}\Ox_Z(D)(U_i \cap U_j)\arrow{r} \arrow{d} & \ldots \\
{} & 0 & 0 
\end{tikzcd}\end{equation}

\begin{lemma} Consider \v Cech $0$-cochains $(v_i) \in \prod_{i}\Ox_X(D)(U_i)$. Call a $0$-cochain $(v_i)$ compatible (with respect to $Z$) if for each $i < j$ there exists $q_{i,j} \in \Ox_X(D-E)(U_i \cap U_j)$ such that $v_i|_{U_i \cap U_j} - v_j|_{U_i \cap U_j} = q_{i,j}r$. Call compatible $0$-cochains $(v_i), (v'_i)$ equivalent, if for each $i$, there exists $q_i \in \Ox_X(D-E)(U_i)$ such that $v'_i - v_i = q_i r$. 
\begin{enumerate}
    \item[(i)] Elements $v\in \R(Z)_D$ are naturally in bijection with equivalence classes of compatible $0$-cochains
    $$(v_i)\in\prod_{i}\Ox_X(D)(U_i).$$ 
    Under this equivalence, the restriction of a section $g \in \R(X)_D$ corresponds to the \v Cech $0$-cocycle $(g|_{U_i}) \in \prod_{i}\Ox_X(D)(U_i)$. 
    \item[(ii)] Multiplication in the algebra $\R(Z)$ is compatible with multiplication of cochains: if $D,D' \in \pic$ and $v \in \R(Z)_D$, $v' \in \R(Z)_{D'}$ are represented by compatible $0$-cochains $(v_i), (v'_i)$ respectively, then $v v' \in \R(Z)_{D+D'}$ is represented by the compatible $0$-cochain $(v_i v'_i)$.
    \end{enumerate}\label{lemma_cocycle}
\end{lemma}
\begin{proof} These statements follow from considering the first column of~\eqref{complex}, recalling that the global section functor on sheaves is exact over affine schemes. Note that a $0$-cochain $(v_i)$ restricts to a $0$-cocycle on $Z$ if and only if it is compatible.
\end{proof}

We are interested in cases when there are equivalence classes of compatible $0$-cochains $(v_i)$ that do not come from the restriction of a section $g\in \R(X)$, so that $(v_i)$ is not a $0$-cocycle.
For $D\in\Pic$, in order to find a complementary vector space to $\iota^*\R(X)_D$ in $\R(Z)_D$, we need to complete the first three steps below.
\begin{enumerate}
    \item Give an explicit description of the first sheaf cohomologies of $\Ox_X(D-E)$ and $\Ox_X(D)$ and of the map $r_{D*}$ of~\eqref{seq:r*}. 
    \item Describe the kernel $N_D=\ker r_{D*}$ as a subspace of $H^1(X, \Ox_X(D-E))$.
    \item Choose a right inverse $\epsilon\colon N_D\to \R(Z)_D$ of $\delta$ that lifts elements of $N_D$ to sections on~$Z$.
    \item Study the ring structure of the direct sum of the resulting spaces to reconstruct the whole structure of~$\R(Z)$. 
\end{enumerate}
Steps (1) and (2) have to be completed in particular cases of interest by explicit calculation. Finding a map $\epsilon$ in step (3) comes down to untangling the connecting homomorphism $\delta$. Given $s \in N_D$ for some $D \in \pic$, let 
\[(s_{i,j})\in \prod_{i<j}\Ox_X(D-E)(U_i \cap U_j)\] be a representative $1$-cocycle. Then $(r \cdot s_{i,j})$ is cohomologically trivial, thus the image of a $0$-cochain \[(v_i) \in \prod_{i}\Ox_X(D)(U_i),\] which is compatible since $v_i|_{U_i \cap U_j} - v_j|_{U_i \cap U_j} = s_{i,j}r$ for each $i < j$. This compatible $0$-cochain corresponds via Lemma~\ref{lemma_cocycle} to an element $v \in \R(Z)_D$ that satisfies $\delta v = s$. Note that the element $v$ is only well-defined up to an element in $\iota^* \R(X)_D$, since $(v_i)$ is well-defined only up to an element of $\ker d^0 = \R(X)_D$. To construct a map $\epsilon$, we need to make compatible choices so that we indeed end up with a well-defined lifting map
\[
\epsilon \colon  N_D  \to  \R(Z)_D,
\]
a right inverse to $\delta$. To complete step (4), further arguments are needed that will be based on Lemma~\ref{lemma_cocycle}(ii). 

To conclude the general discussion, let us make one further remark. Suppose that $v \in \R(Z)_D$ is represented by a compatible $0$-cochain $(v_i)$ via Lemma~\ref{lemma_cocycle}(i). For fixed $i \in I$, we have $v_i \in \Ox_{X}(U_i)(D)$, and thus $v_i = \frac{h_i}{g_i}$ for some regular functions $h_i,g_i$ on $X$ with $g_i$ not vanishing on $U_i$. Identifying $h_i$ and $g_i$ with their restrictions on $Z$, the section $g_i v_i - h_i$ is zero on $Z \cap U_i$, which is open and dense in $Z$. For each $i$, we thus obtain an equation
\[g_i v - h_i = 0\]
in $\R(Z)$. This explains the relevance of localisations of~$\R(X)$ for the problem, which is the basis of the approach taken in~\cite{lafaceetal}. 

\section{Cox rings of hypersurfaces in a product of projective spaces}
\subsection{Basics}
Fix $n \geqslant 3$ and let $X = \p^1 \times \p^n$. We use homogeneous coordinates $x_0, x_1$ and $y_0, \dots, y_n$ on $\p^1$ and $\p^n$ respectively. If $p_1, p_2$ are the projection maps to the respective factors $\p^1, \p^n$, then $\Pic X\cong \Z^2$ is generated by the pullbacks $p_1^* \Ox_{\p^1}(1)$ and $p_2^* \Ox_{\p^n}(1)$ respectively. The Cox ring of $X$ is \[\coxX=\R(X) = \C[x_0,x_1,y_0,\dots,y_n],\] with the $\Z^2$-grading given by $\deg(x_i) = (1,0)$ and $\deg(y_j) = (0,1)$. The ring \[R=\R(\p^n) = \C[y_0, \dots, y_n]\] is naturally a $\Z^2$-graded subring of~$S$, making~$S$ into a $\Z^2$-graded $R$-module.

Fix 
$n \geq 3$ and $d,e \geq 1$. Let $r \in \coxX_{(d,e)}$ and
\[Z=\left\{ r= 0\right\} \subset X\] 
be the corresponding hypersurface of bidegree $(d,e)$ in $X$. Then $Z$ is defined by the bihomogeneous section 
\begin{equation}
    r = g_0 x_0^d + g_1 x_0^{d-1} x_1 + \ldots + g_d x_1^d
\end{equation}
for some $g_0, g_1, \ldots, g_d \in R_e$. If $Z$ is nonsingular, the Lefschetz hyperplane theorem shows that the inclusion $\iota\colon Z\hookrightarrow X$ induces an isomorphism of Picard groups $\iota^*\colon\Pic X \cong \Pic Z \cong\Z^2$ that will be denoted $\Pic$. We will also consider a singular case, where we will comment on the relation between the Picard groups explicitly.

We require an ordered open cover $\mathcal{W}$ of $X$. For $i = 0,1$, let $U_i \subset \p^1$ be the standard affine open defined by $x_i \neq 0$. For $j = 0,\dots,n$, let $V_i \subset \p^n$ be the standard affine open defined by $y_j \neq 0$. Set $W_{ij} = U_i \times V_j$ for each $i,j$ and let $\mathcal{W} = \{W_{ij}\}_{i,j}$, with the lexicographic ordering. If $ij < kl$, then denote $W_{ij} \cap W_{kl}$ by $W_{ij,kl}$.

\subsection{The degeneracy locus of the second projection}
Consider the projection $p_2\colon X\to\p^n$ restricted to $Z$. This is generically a $d$-to-one cover, but it has a degeneracy locus 
\[ Y = \{ g_0=g_1=\cdots = g_d=0\}\subset\p^n,\]
over which the fibres are isomorphic to~$\p^1$. Let 
\[ I_Y = \langle g_0, g_1, \ldots g_d\rangle \lhd R = \C[y_0, \ldots, y_n]\]
be the corresponding ideal. The quotient $R/I_Y$ has a finite free $R$-module resolution of the form 
\begin{equation} \label{eq:sequence_for_T} \ldots \longrightarrow R^{\gamma}  \stackrel{C}\longrightarrow R^{\beta}\stackrel{B}\longrightarrow R^{d+1}\stackrel{A}\longrightarrow R \longrightarrow R/I_Y\longrightarrow 0,\end{equation}
where we suppressed degrees, with $A=(g_0, g_1,\ldots , g_d)$, and $B,C$ matrices of homogeneous elements of the appropriate sizes and degrees. 

Of particular interest to us is the case where $Z$ is regular with index sequence $\{i_0,\ldots, i_l\}$. It is easy to check that there are always such choices with $Z \subset X$ nonsingular. Then \eqref{eq:sequence_for_T} is the sum of the standard Koszul complex for $\{g_{i_0},\dots,g_{i_l}\}$ and a trivial complex. The locus $Y \subset \p^n$ is a complete intersection of codimension $l+1$, and the inverse image $E = p_2^{-1}(Y) \subset X$ has codimension $l$ in $Z$. In the extreme case, where $l=1$ and $r = g_0 x_0^d + g_d x_1^d$, the variety $E$ is a divisor ruled over $Y$. 

\subsection{The case $d=0$}
The case $d=0$ is somewhat degenerate, but it will be useful to include as it provides a convenient case for our induction. For $d=0$, $Z \cong \p^1 \times Z_1$, with $Z_1= \{g_0=0\} \subset \p^n$. If moreover $n \geq 4$, then clearly $\R(Z) \cong \R(\p^1\times\p^n)/\langle g_0\rangle$, compatibly with the statement of Theorem~\ref{thm:mainintro}. However, if $(n,d) = (3,0)$, then $\Pic X$ and $\Pic Z$ may not be isomorphic and Theorem~\ref{thm:mainintro} may fail to hold. Indeed, for example if $(n,d,e) = (3,0,2)$ and $g_0$ is a general quadric, then \[Z\cong \p^1\times Q \cong \p^1 \times \p^1 \times \p^1\]
with $Q\subset\p^3$ a general quadric surface, so that there is a $\Z^3$-graded isomorphism
    \[
    \R(Z) \cong \C[x_0,x_1,u_0,u_1,v_0,v_1].
    \]
However, the Veronese subring of $\R(Z)$ associated to the subgroup $\iota^* \R(X) \subset \R(Z)$, which is generated by $(1,0,0)$ and $(0,1,1)$, takes the form \eqref{presentationR}, which in this case is
 \[  \bigoplus_{D\in\iota^*\!\Pic X}\R(Z)_D \cong k[x_0,x_1,y_0,y_1,y_2,y_3]/(g_0). \]

\subsection{Describing the maps on first cohomology}

The first \v Cech cohomology groups of $X=\p^1\times\p^n$ can be written using the Künneth formula, or computed explicitly using the cover $\mathcal{W}$. This gives the following result, completing Step (1) of 
our plan sketched in~\ref{subsec:cech}. 
\begin{lemma} \label{lem:H1}\begin{enumerate} \item[(i)] We have natural $\Z^2$-graded vector space isomorphisms 
\begin{equation}\label{eq:isoH1}
    \bigoplus_{(a,b) \in \Z^2} H^1(X,\Ox_X(a,b)) \cong \frac{1}{x_0 x_1}R[x_0^{-1}, x_1^{-1}] \cong S[x_0^{-1}, x_1^{-1}]/L,
\end{equation}
where recall $R=\C[y_0,\ldots, y_n]$, $S=\C[x_0,x_1, y_0,\ldots, y_n]$, and $L$ is the vector subspace of 
$S[x_0^{-1}, x_1^{-1}]$ generated by all monomials in $x_i, y_j$ in which $x_0$ or $x_1$ appear with non-negative exponent.
\item[(ii)] The map
\[
r_{*} : \displaystyle\bigoplus_{D\in\Pic} H^1(X, \Ox_X(D-E)) \longrightarrow \displaystyle\bigoplus_{D\in\Pic} H^1(X, \Ox_X(D))
\]
can be identified with the map 
\[\begin{array}{rcl}
r_{*} : S[x_0^{-1}, x_1^{-1}]/L & \longrightarrow & S[x_0^{-1}, x_1^{-1}]/L\\
f & \mapsto &  r \cdot f \mod L
\end{array}
\]
that multiplies an element $f\in S[x_0^{-1}, x_1^{-1}]/L\cong\frac{1}{x_0 x_1}R[x_0^{-1}, x_1^{-1}]$ by $r$, and ignores those terms in which $x_0$ or $x_1$ occur with non-negative exponent. 
\end{enumerate}
\end{lemma}
\begin{proof} This first isomorphism in~\eqref{eq:isoH1} is clear from the K\" unneth formula. This identifies an element $f \in \frac{1}{x_0 x_1}R[x_0^{-1}, x_1^{-1}]_{(a,b)}$ with class of the $1$-cocycle $(s_{ij,kl})$ representing an element of $H^1(X,\Ox_X(a,b))$, where $s_{ij,kl} = f$ if $i \neq k$ and $s_{ij,kl} = 0$ if $i = k$. The second isomorphism in~\eqref{eq:isoH1} is immediate. The description in (ii) is also clear from the explicit cocycle description. 
\end{proof}
Note that the spaces appearing in the isomorphism~\eqref{eq:isoH1} naturally have the structure of a graded $R$-module, which we will use in our subsequent arguments. 

We move on to Step (2) of our plan. Since $H^1(X,\Ox_X(a,b))$ is only non-zero for $a<0$, we will use the index $-a$ rather than $a$. For fixed $a > 0$, there is an isomorphism of $R$-modules
\begin{equation}\label{eq:isodirectsum} \begin{array}{rcl}R^{a-1}& \stackrel{\sim}\longrightarrow& \displaystyle\bigoplus_{b \in \Z} 
\frac{1}{x_0 x_1}R[x_0^{-1},x_1^{-1}]_{(-a,b)} \\
(f_1, \ldots, f_{a-1}) & \mapsto & f=\displaystyle\sum_{k=1}^{a-1} \frac{f_k}{x_0^{a-k} x_1^k}.
\end{array}\end{equation}
Using Lemma~\ref{lem:H1}(ii), we obtain the following description. 
\begin{proposition}\label{prop:ker} For $D=(-a,b)\in\Pic$, the kernel $N_{(-a,b)}=\ker r_{D*}$ can only be nonzero if $-a \leq d-2$. 
\begin{enumerate}\item[(i)]
For~$-1 \leq -a \leq d-2$, we have a graded isomorphism
\begin{equation*}
    \bigoplus_{b \in \Z} N_{(-a,b)} \cong R^{a+d-1}[e], 
\end{equation*}
where $[e]$ denotes a shift in grading so that, for example, 
$N_{(d-2,b)} \cong R_{b-e} \cdot \frac{1}{x_0 x_1}$
for $b \in \Z$. 
\item[(ii)]
For~$-a<-1$, we have a graded isomorphism
\begin{equation*}
    \bigoplus_{b \in \Z} N_{(-a,b)} \cong \ker A_a[e],
\end{equation*}
where $A_a$ is the $R$-module homomorphism
\[
\begin{array}{rcccl}
  A_a & \colon & R^{a+d-1} & \longrightarrow & R^{a-1} 
  \\  && (f_k)_{k=1}^{a+d-1} & \longmapsto &  (g_0 f_{k} + g_1 f_{k+1} + \ldots + g_d f_{k+d})_{k=1}^{a-1}
\end{array}
\]
defined by the matrix
\begin{equation*}
         \begin{pmatrix}
            g_0 & g_1 & g_2 & \cdots & g_d & 0 & 0 & \cdots \cdots & 0 & 0 & 0 & \cdots & 0 & 0\\
            0 & g_0 & g_1 & \cdots & g_{d-1} & g_d & 0 & \cdots \cdots & 0 & 0 & 0 & \cdots & 0 & 0\\
            \ & \vdots & \ & \ddots & \ & \vdots & \ & \ddots & \vdots & \ & \ & \ddots & \ & \vdots\\
            0 & 0 & 0 & \cdots & 0 & 0 & 0 & \cdots \cdots & 0 & g_0 & g_1 & \cdots & g_d & 0 \\
            0 & 0 & 0 & \cdots & 0 & 0 & 0 & \cdots \cdots & 0 & 0 & g_0 & \cdots & g_{d-1} & g_d
        \end{pmatrix}.
    \end{equation*}
\end{enumerate}
\end{proposition}

The first non-trivial example of a map $A_a$ is
\begin{equation*}\begin{array}{rcccl}
    A_2 & \colon & R^{d+1} & \longrightarrow & R \\
     && (f_1, \ldots, f_{d+1}) & \mapsto & g_0 f_{1} + g_1 f_{2} + \ldots + g_d f_{d+1}
\end{array}     
\end{equation*}
which is the same as the map $A$ in the complex~\eqref{eq:sequence_for_T}, and is the first map in the Koszul complex associated with the sequence $\{g_0, \dots, g_d\}$ in $R$. Its kernel, the {\em syzygies} between the defining polynomials $\{g_0, \dots, g_d\}$, give us elements in $\R(Z)_{(-2,b)}$ for different values $b$. 

\subsection{Lifting cohomology elements to sections in the Cox ring} \label{sec:lifting} In this subsection, we describe how to perform Step~(3) from Section~\ref{subsec:cech} in our context. For $D=(-a,b)\in\Pic$, recall the standard exact sequence~\eqref{seq:short}
\[
0\longrightarrow \iota^*\R(X)_{(-a,b)}\stackrel{}\longrightarrow \R(Z)_{(-a,b)} \stackrel{\delta}\longrightarrow N_{(-a,b)}\longrightarrow 0.
\]
We need to construct a lifting map $\epsilon\colon N_{(-a,b)} \rightarrow \R(Z)_{(-a,b)}$ that is a right inverse to $\delta$. By Proposition~\ref{prop:ker}, $N_{(-a,b)}$ can only be non-zero if $-a \leq d-2$; fix such an $a$. Let $f \in N_{(-a,b)}$. Then $u = r \cdot f$ is a degree $(-a,b)$ element in $S[x_0^{-1}, x_1^{-1}]$ such that in each non-zero term, at least one of $x_0, x_1$ appears with non-negative (possibly zero) exponent. Let $u^{(0)}$ be the sum of terms in $u$ where $x_0$ has negative exponent, and let $u^{(1)}$ be the sum of terms in which $x_1$ has negative exponent. Let $u^{(2)}$ be the sum of remaining terms of $u$. Then 
\begin{equation} \label{eq:split}
u  = u^{(0)}+u^{(1)}+u^{(2)}
\end{equation}
with
\[ \begin{array}{c}
    u^{(i)} \in \Ox_X(-a,b)(W_{ij})  \mbox{ for all } j \mbox{ and } i = 0,1, \\
    u^{(2)} \in S_{(-a,b)}.
\end{array}
\]
Then the $0$-cochain $(v_{ij})$ on $X$ given by 
\begin{equation}\label{eq_nonuniquelift}
v_{ij} =
 \left\{
  \begin{array}{ll}
			 u^{(0)} + \frac{1}{2} u^{(2)} & \mbox{if } i = 0, \\
			 -u^{(1)} - \frac{1}{2} u^{(2)} & \mbox{if } i = 1, 
		\end{array}
		\right.
\end{equation}
is compatible, thus by Lemma~\ref{lemma_cocycle} gives a section $v \in \R(Z)_{(-a,b)}$. The considerations at the end of Section~\ref{subsec:cech} translate into 
\begin{proposition} \label{prop:epsilon}
    The map 
\[\begin{array}{rcccl} \epsilon& \colon &  N_{(-a,b)} & \to & \R(Z)_{(-a,b)}\\
&& f & \mapsto & v
\end{array}
\]
defines a right inverse to the connecting homomorphism $\delta\colon \R(Z)_{(-a,b)} \to N_{(-a,b)}$.
\end{proposition}
Note that if $-a<0$ then necessarily $u^{(2)}=0$, and the lifting map $\epsilon$ is unique. If however $u^{(2)}\neq 0$, then 
it is not; we make a ``symmetric'' choice for our map $\epsilon$. 
\begin{remark}
    We note that the choice of $\epsilon$ is compatible with the $R$-module structures on $\bigoplus_{b\in \Z} N_{(-a,b)}$ and $\bigoplus_{b \in \Z} \R(Z)_{(-a,b)}$ for each $a \in \Z$. We can thus view $\epsilon$ as an $R$-module map from $\bigoplus_{b\in \Z} N_{(-a,b)}$ to $\bigoplus_{b \in \Z} \R(Z)_{(-a,b)}$ for each $a$, which we will use to prove the next Proposition.
\end{remark}

\subsection{Finding some sections in the Cox ring}
In this subsection, we complete Step~(3) in our procedure for all degrees $(-a,b)$ with $-1 \leq -a \leq d-2$, under the only assumption that $Z$ is nonsingular. We find $d$ new sections in the Cox ring of~$Z$ in degree~$(-1,e)$, and use these to describe $\R(Z)_{(-a,b)}$ for all degrees $(-a,b)$ with $-a \geq -1$.  

\begin{proposition}\label{prop:z_k}
\begin{enumerate} \item[(i)] There exists sections $z_1, \dots, z_d \in \R(Z)_{(-1,e)}$, linearly independent over $R$, that satisfy the $(d+1)$ equations
\begin{equation}
    x_1 z_1 + g_0 = x_1 z_{2} + g_{1} - x_0 z_{1} = \dots = x_1 z_{d} + g_{d-1} - x_0 z_{d-1} = g_d - x_0 z_d = 0.
    \label{eq:coxz}
\end{equation}
\item[(ii)]The subalgebra $\R(Z)_{z_k}$ of $\R(Z)$ generated by $x_0, x_1, y_0, \dots, y_n, z_1, \dots, z_d$ contains $\R(Z)_{(-a,b)}$ for all $-a \geqslant -1, b \in \Z$.\end{enumerate}
\label{prop:coxz1}
\end{proposition}

\begin{proof}
By Proposition~\ref{prop:ker}(i) we have
\begin{equation}\label{kerdecomp}
\bigoplus_{b\in \Z} N_{(-1,b)} = \bigoplus_{k = 1}^d R[e]\cdot \frac{1}{x_0^{d-k+1} x_1^k} \cong R[e]^d.
\end{equation}
For $1 \leq k \leq d$, let $f_k = \frac{1}{x_0^{d-k+1}x_1^k}$ be the $k^{\text{th}}$ basis vector in the decomposition (\ref{kerdecomp}). The short exact sequence~\eqref{seq:short} implies $\delta$ and $\epsilon$ are inverse isomorphisms for $a = -1$. Thus setting $z_k = \epsilon f_k$ gives us $d$ $R$-linearly independent sections $z_1, \dots, z_d \in \R(Z)_{(-1,e)}$. Following the definition of $\epsilon$, each $z_k$ is given by the following $0$-cochain $(z_{k,ij})$ on $X$:

\begin{equation*}
z_{k,ij} =
 \left\{
  \begin{array}{ll}
			 \sum_{c-k \geqslant 0} g_c \frac{x_1^{c-k}}{x_0^{c-k+1}} & \mbox{if } i = 0, \\
			 -\sum_{c-k < 0} g_c \frac{x_0^{k-c-1}}{x_1^{k-c}} & \mbox{if } i = 1. \\
		\end{array}
		\right.
\end{equation*}
Since for any choice of $(i,j)$ we have $x_1 z_{k+1,ij} + g_k - x_0 z_{k,ij} = 0$, we conclude that, for $k = 0,1,\dots,d$,
\begin{equation*}
    x_1 z_{k+1} + g_k - x_0 z_{k} = 0,
\end{equation*}
where the undefined variables $z_0$ and $z_{d+1}$ are set to be zero. This proves Proposition \ref{prop:coxz1}(i).

For part (ii), note that $\iota^*\R(X)$ is generated as an algebra by the sections $x_0,x_1,y_0,\dots,y_n$ since this is true for $\R(X)$. To prove the statement it thus suffices to prove that each section in $\epsilon(N_{(-a,b)})$, for each $-a \geq -1$, is in $\R(Z)_{z_k}$. Part (i) covers the $a=-1$ case. Proposition~\ref{prop:ker} gives that $N_{(-a,b)}= 0$ for $-a > d-2$ and that
\begin{equation*}
\bigoplus_{b\in \Z} N_{(-a,b)} = \bigoplus_{k = 1}^{a+d-1} R[e]\cdot \frac{1}{x_0^{a+d-k} x_1^k} \cong R[e]^{a+d-1}.
\end{equation*}
for $0 \leq -a \leq d-2$. Fix $a$ in this latter range and for each $1 \leq k \leq a+d-1$ set $f_{-a,k} = \frac{1}{x_0^{a+d-k}x_1^k}$. These $f_{-a,k}$ give an $R$-basis for $\bigoplus_{b \in \Z} N_{(-a,b)}$. Since $\delta$ and $\epsilon$ are $R$-module maps, we only need to prove that each $v_{-a,k} = \epsilon f_{-a,k}$ is in $\R(Z)_{z_k}$ to conclude our proof. Computing the $0$-cochains associated to the $v_{-a,k}$ and comparing to the $0$-cochains associated to the generators of $\R(Z)_{z_k}$, we find for example the equations
\[
v_{-a,k} = z_k x_0^{1-a} - \frac{1}{2} \sum_{0 \leq c-k \leq -a} g_c x_0^{-a-c+k} x_1^{c-k}
\]
for each $0 \leq -a \leq d-2$ and $1 \leq k \leq a+d-1$. Then each $v_{-a,k}$ is in $\R(Z)_{z_k}$ and we are done.
\end{proof}

We deduce

\begin{corollary} \label{cor:partialdescription} Let
\[T = \C[X_0,X_1,Y_0,\dots,Y_n,Z_1,\dots,Z_d]\]
be the $\Z^2$-graded $\C$-algebra with the generators having degrees $(1,0)$, $(0,1)$ and $(-1,e)$, respectively. Let $I\lhd T$ be the ideal 
\[ I = \langle X_1 Z_1 + g_0(Y_i), X_1 Z_2 + g_1(Y_i) - X_0 Z_1,\ldots, g_d(Y_i) - X_0 Z_d\rangle\]
with generators corresponding to the $d+1$ equations in (\ref{eq:coxz}).
Then the algebra homomorphism 
\[ \phi : T \rightarrow \R(Z)
\]
defined by $X_i\mapsto x_i$, $Y_i\mapsto y_i$, $Z_i\mapsto z_i$ has kernel $K = \ker \phi$ containing~$I$, inducing isomorphisms $(T/I)_{(a,b)}\cong \R(Z)_{(a,b)}$ whenever $a \geqslant -1$.
\end{corollary}
\begin{proof}
By Proposition \ref{prop:coxz1}, the ideal $I$ is contained in $K = \ker \phi$. To prove the corollary, it suffices to show that $(K/I)_{a,b} = 0$ for $a \geqslant -1$.
Firstly, let $r' = g_0(Y_j)X_0^d + \dots + g_d(Y_j)X_1^d$, which maps via $\phi$ onto the defining equation $r$ of $Z$, so that $r' \in K$. Since
\begin{align*}
    r = &X_0^d \left(X_1 Z_1 + g_0(Y_j)\right) + X_0^{d-1} X_1 \left(X_1 Z_2 + g_1(Y_j) - X_0 Z_1\right) + \\ &\dots + X_0 X_1^{d-1} \left(X_1 Z_{d-1} + g_{d-1}(Y_j) - X_0 Z_d\right) + X_1^d \left(g_d(Y_j) - X_0 Z_d\right),
\end{align*}
we have $r' \in I$. Let $f = f(X_i,Y_j,Z_k)\in K$ be a homogeneous degree $(a,b)$ polynomial, for $a \geqslant -1$. We prove that the image $\overline{f}$ of $f$ in $K/I$ is zero by considering cases in $a$. Note that any term in $f$ is a monomial in the $Y_j$ multiplied by a monomial $M = X_0^{\alpha_0} X_1^{\alpha_1} Z_1^{\beta_1}\dots Z_d^{\beta_d}$ such that $\alpha_0 + \alpha_1 - \beta_1 - \cdots - \beta_d = a$.

If $a > d-2$, the generating equations of $I$ allow us to rewrite such an $M$ purely as a polynomial in the $X_i$ and $Y_j$ modulo $I$. Without changing the class $\overline{f}$, we replace $f$ with a polynomial $f(X_i,Y_j)$ independent of the $Z_k$. Since $\R(Z)_{(a,b)} = \C[x_i,y_j]_{(a,b)} / r\cdot \C[x_i,y_j]_{(a-d,b-e)}$ by Proposition \ref{prop:ker}, that $f \in K$ implies that $r$ divides $f(x_i,y_j)$ in $\C[x_i,y_j]$, so that $r'$ divides $f$ and thus $f \in I$ and $\overline{f} = 0$.

If $-1 \leq a \leqslant d-2$, the procedure is similar. The generating equations of $I$ allow us to rewrite $M$ as a $\C[Y_j]$-linear combination of $X_0^{a+1}Z_1,\dots,X_0^{a+1}Z_{d-a-1}$ plus a polynomial independent of the $Z_k$ modulo $I$. We can thus replace $f$, without changing $\overline{f}$, with a polynomial of the form
\begin{equation*}
    f(X_i,Y_j,Z_k) = f_1(Y_j)X_0^{a+1} Z_1 + \dots + f_{d-a-1}(Y_j)X_0^{a+1}Z_{d-a-1} + p(X_i,Y_j).
\end{equation*}
Following the proof of Proposition \ref{prop:coxz1}, we have
\begin{equation*}
    \phi(f) = f_1(y_j) v_{a,1} + \dots + f_{d-a-1}(y_j) v_{a,d-a-1} + q(x_i,y_j)
\end{equation*}
for some $q \in \C[x_i,y_j]_{(a,b)}\cong\iota^*\R(X)_{(a,b)}$. But $q, v_{a,1}, \dots, v_{a,d-a-1}$ are $R$-linearly independent, and so $\phi(f) = 0$ implies $f_1= \dots = f_{d-a-1} = 0$, and in turn $q = p = 0$. Then $\bar f=0$ as required.
\end{proof}

Proposition \ref{prop:coxz1}(ii) says that any Cox ring generators other than $x_i, y_j, z_k$ must be in degree $(a,b)$ with $a \leq -2$. Corollary \ref{cor:partialdescription} tells us that any new relations between generators must also be in degree $(a,b)$ with $a \leq -2$.

\subsection{Multiplying sections}

Proposition~\ref{prop:coxz1} and Corollary~\ref{cor:partialdescription} give a full picture of $\R(Z)$ in degrees $(-a,b)$ with $-a \geq -1$, for an arbitrary set of defining polynomials $\{g_c\}$ as long as $Z$ is non-singular. The computation of the remaining sections was reduced in Proposition~\ref{prop:ker} to the description of the kernels of the maps $A_a$. For general $\{g_c\}$, we are unable to give a full description; even the structure of $\ker A_2$ is hard to describe explicitly in general. In this section, we prepare the ground for further computations by utilising the identifications of each $\ker A_a$ with a submodule of $R^{a+d-1}$ to interpret multiplication by $x_0, x_1$ as well as $z_1,\dots, z_d$ in a succinct way. 

First, consider the multiplication map
\[ \cdot x_i \colon \R(Z)_{(-a,b)} \to \R(Z)_{(-a+1,b)}. 
\]
Composing this map with our lifting map $\epsilon$ on the right, and the projection map $\delta$ on the left, we obtain a map
\[\tilde x_i : \ker A_{a} \rightarrow \ker A_{a-1}.\]
\begin{proposition}
    Suppose that $-a \leqslant -2$. For $i=0,1$, the map $\tilde x_i : \ker A_{a} \rightarrow \ker A_{a-1}$ truncates the $(a+d-1)$-tuple $(f_m)_{m=1}^{a+d-1}$ on the right, respectively on the left, by one term.
    \label{xmult}
\end{proposition}
\begin{proof}
Fix $-a \leqslant -2, b \in \Z$ and let $(f_m)_{m=1}^{a+d-1} \in \ker A_a$,
identified with $f = \sum_{m=1}^{a+d-1} \frac{f_m}{x_0^{a+d-m} x_1^m}$.
Using Section~\ref{sec:lifting}, we obtain the lifting $w=\epsilon f$ defined by the degree $(-a,b)$ $0$-cochain $(w_{ij})$ given by
\begin{eqnarray*}
w_{0j} & = & f_1 g_d \frac{x_1^{d-1}}{x_0^{a+d-1}} + (f_1 g_{d-1} + f_2 g_d)\frac{x_1^{d-2}}{x_0^{a+d-2}} + \dots + (f_1 g_1 + \dots + f_d g_d) \frac{1}{x_0^a}, \\
w_{1j} & =  &  -f_{a+d-1} g_0 \frac{x_0^{d-1}}{x_1^{a+d-1}} - (f_{a+d-2} g_{0} + f_{a+d-1} g_1)\frac{x_0^{d-2}}{x_1^{a+d-2}} - \dots - (f_a g_0 + \dots + f_{a+d-1} g_{d-1}) \frac{1}{x_1^a}.
\end{eqnarray*}
If $^-f$ is identified with the element of $\ker A_{a-1}$ obtained from truncating $(f_m)_{m=1}^{a+d-1}$ on the right, namely then $(f_m)_{m=1}^{a+d-2}$, then repeating this procedure for $^-w = \epsilon (^-f)$, we notice that $x_0 w_{0j} = ^-w_{0j}$ for each $j$, thus $x_0 w$ and $^-w$ agree on the dense open sets $Z \cap W_{0j}$ and therefore $x_0 w = ^-w$. If $f^-$ is the truncation on the left with corresponding section $w^- = \epsilon (f^-)$, then similarly $x_1 w_{1j}$ and $w^-_{1j}$ agree on the sets $Z \cap W_{1j}$ thus also $x_1 w = w^-$. 
\end{proof}

We now turn our attention to multiplying by $z_1, \dots, z_d$. 

\begin{proposition}
    Suppose that $-a \leqslant -1$. For each $k = 1,\dots,d$, the map $\tilde z_k : \ker A_{a} \rightarrow \ker A_{a+1}$, corresponding to multiplication by $z_k$, is given as follows. Suppose that $(f_m)_{m=1}^{a+d-1} \in \ker A_a$. Then $\tilde z_k (f_m)_{m=1}^{a+d-1}$ is given by $(f'_m)_{m=1}^{a+d}$, where
    \begin{align*}
        f'_m =& f_m g_k + \dots + f_{m+d-k} g_d, \quad m = 1,\dots,a+k-1, \\
        f'_m =& -f_{m-k} g_0 - \dots - f_{m-1} g_{k-1}, \quad m = k+1,\dots, a+d.
    \end{align*}
    If $-a = -1$, then this is specifies all $f'_m$. If $-a < -1$ then, for $m = k+1, \dots, a+k-1$, the two given definitions of $f'_m$ agree, as $(f_m)_{m=1}^{a+d-1} \in \ker A_a$.
    \label{zmult}
\end{proposition}

\begin{proof}
    Fix $-a \leqslant -1$. One checks that the maps as defined are indeed $R$-module homomorphisms with image in $\ker A_{a+1}$. Suppose that $(f_m)_{m=1}^{a+d-1} \in \ker A_a$ and let $\tilde z_d (f_m)_{m=1}^{a+d-1} = (f'_m)_{m=1}^{a+d}$. Then using Proposition~\ref{xmult} for the map $\tilde x_0$, we obtain
    \begin{equation*}
        \tilde{x}_0 \tilde{z}_d (f_m)_{m=1}^{a+d-1} = (f'_m)_{m=1}^{a+d-1}.
    \end{equation*}
    But the first equation in (\ref{eq:coxz}) tells us that the map $\tilde x_0 \tilde z_d$ is the same as multiplication by $g_d$, giving $f'_m$ as defined in the Proposition for the range $m = 1, \dots, a+d-1$. Since $(f'_m)_{m=1}^{a+d} \in \ker A_{a+1}$, the element $f'_{a+d}$ is uniquely determined by $f'_{a},\dots,f'_{a+d-1}$ and must also be given as in the Proposition. With the result proved for $k=d$, we proceed with a downwards induction in $l$. Suppose the formula is true for $k = d-k'$ and let $\tilde z_{d-k'-1} (f_m)_{m=1}^{a+d-1} = (f'_m)_{m=1}^{a+d}$. By the corresponding equation in \eqref{eq:coxz} we have
    \begin{equation*}
        (f'_m)_{m=1}^{a+d-1} = \tilde{x}_0 \tilde{z}_{d-k'-1} (f_m)_{m=1}^{a+d-1} = (f_m g_{d-k'-1})_{m=1}^{a+d-1} +\tilde{x}_1 \tilde{z}_{d-k'}(f_m)_{m=1}^{a+d-1}.
    \end{equation*}
    Using Proposition~\ref{xmult} for $\tilde x_1$ and the induction hypothesis we obtain the claimed formula for $f'_m$ for $m = 1,\dots,a+d-1$. The corresponding formula for $f'_{a+d}$ again follows as this is uniquely determined by $f'_{a},\dots,f'_{a+d-1}$.
\end{proof}

\subsection{The main theorem}

In this subsection, we prove Theorem~\ref{thm:mainintro}. 


\begin{theorem}\label{thm:general} Fix $n \geq 3$ and $d,e \geq 1$. Let $Z \subset X$ is non-singular, defined by a degree $(d,e)$ equation $r \in \R(X)$. Suppose that $Z$ is regular with index sequence $\{i_0,\dots,i_l\}$, where $1 \leq l \leq d$. Then there are sections $w_k \in \R(Z)_{(i_{k-1} - i_k,e)}$ for each $1 \leq k \leq l$ satisfying the equations
    \begin{equation}\label{eq:coxgen}
    x_1^{i_1 - i_0} w_1 + g_{i_0} = x_1^{i_2 - i_1} w_2 + g_{i_1} - x_0^{i_1 - i_0} w_1 = \dots = g_{i_l} - x_0^{i_l - i_{l-1}} w_l = 0.
   \end{equation}
    Furthermore, these sections, along with the sections $x_0,x_1,y_0, \dots,y_n$ restricted from $X$, generate the Cox ring $\R(Z)$. The ideal of relations among these generators is generated by the equations \eqref{eq:coxgen}. More precisely, let $U = \C[X_0, X_1, Y_0, \dots, Y_n, W_1,\dots,W_l]$ be the free $\Z^2$-graded polynomial ring with $\deg(X_i) = (1,0), \deg(Y_j) = (0,1)$ and $\deg(W_k) = (i_{k-1} - i_k,e)$ for the relevant $i,j,k$, and $J\lhd U$ the homogeneous ideal
    \begin{equation*}
       J = \langle X_1^{i_1 - i_0} W_1 + g_{i_0}(Y_j), X_1^{i_2 - i_1} W_2 + g_{i_1}(Y_j) - X_0^{i_1 - i_0} W_1,  \dots , g_{i_l}(Y_j) - X_0^{i_l - i_{l-1}} W_l\rangle. 
    \end{equation*}
    Then there is a surjective homomorphism of $\Z^2$-graded algebras
    $\psi : U \rightarrow \R(Z)$, inducing an isomorphism $\R(Z) \cong U/J$. In particular, the hypersurface~$Z$ is a Mori Dream Space.
\end{theorem}

\begin{proof}
    The sections $w_k$ are immediately found from the defining equation of $Z$. Indeed, for each $1 \leq k \leq l$,
    \begin{equation}\label{eq:wkdef}        
w_k=  \frac{g_{i_k} x_0^{d-i_k} + \dots + g_{i_l}x_1^{d-i_k}}{x_0^{d-i_{k-1}}} = -\frac{g_{i_0} x_0^{i_{k-1}} + \dots + g_{i_{k-1}}x_1^{i_{k-1}}}{x_1^{i_{k}}} \in \R(Z)_{(i_{k-1}-i_k,e)}
    \end{equation}
    is defined globally on $Z$, and these sections satisfy the equations (\ref{eq:coxgen}) in Theorem~\ref{thm:general}.
    In terms of our identification in Proposition~\ref{prop:ker}, the section $w_k$ for $1 \leq k \leq l$ is associated to the element $(f_m)_{m=1}^{i_k - i_{k-1} + d - 1} \in \ker A_{i_k - i_{k-1}}$ with $f_m = 1$ if $m = i_k$ and all other $f_m$ equal zero. We thus have a well-defined map $\psi\colon U\to \R(Z)$ mapping $X_i, Y_j, W_k$ respectively to the sections $x_i, y_z, w_k\in\R(Z)$, respectively. 

    The sections $z_1,\dots, z_d\in \R(Z)$ from Proposition~\ref{prop:coxz1}(i) are in the image of $\psi$; namely, for each $1 \leq k \leq l$ and each $1 \leq m \leq i_k - i_{k-1}$, it is easy to check that 
    \begin{equation} \label{eq:zinxw}
    z_{i_{k-1} + m} = \psi(Z_{i_{k-1}+m}) \ \mbox{ with } \ Z_{i_{k-1}+m} =X_0^{m-1}X_1^{i_k - i_{k-1} - m} W_k.
    \end{equation}
 So $\im\psi$ contains $\R(Z)_{(-a,b)}$ for all $-a \leq -1$ by Proposition~\ref{prop:coxz1}(ii). Furthermore, the equations \eqref{eq:coxz} become exactly the equations \eqref{eq:coxgen} and any degree $(a,b)$ monomial of $U$ with $a \geq -1$ can be rewritten as a monomial in  $X_0, X_1,Y_0,\dots,Y_n,$ $Z_1,\dots,Z_d$. Thus, by Corollary~\ref{cor:partialdescription}, we only need to consider $\R(Z)$ in degrees $(-a,b)$ with $-a \leq -2$. Surjectivity of $\psi$ in degrees $(-a,b)$ with $-a \leq -1$ is given by Theorem~\ref{regularsequencekernels}. It remains to show that the induced surjective map $\R = U/J \rightarrow \R(Z)$ is in fact an isomorphism.
 
Since $\psi$ is surjective, $Z$ is a Mori Dream Space, and thus $\dim \R(Z) = \dim Z + \rho (Z) = n+2$ by \cite[Proposition 2.9]{MDS}. To conclude the proof of the theorem, it therefore suffices to show that the algebra $\R = U/J$ is integral of dimension $n+2$. 

Localising $\R$ at $X_0$, we have the equations
\[
W_l = \frac{g_{i_l}(Y_j)}{X_0^{i_l - i_{l-1}}}, \dots, W_1 = \frac{X_1^{i_2 -i_1}W_2 + g_{i_1}(Y_j)}{X_0^{i_1 - i_0}}
\]
in $\R_{X_0}$, implying the surjectivity of 
\[
\C[X_0,X_1,Y_0,\dots,Y_n]_{X_0} \rightarrow \R_{X_0}.
\]
Substituting the above equations into $X_0^{i_1 - i_0} W_1 + g_{i_0}(Y_j) = 0$ in $\R_{X_0}$ we obtain $\frac{r(X_i,Y_j)}{X_0^d} = 0$ and thus
\begin{equation}\label{eq:Rlocalised0}
    \R_{X_0} = \left(\C[X_0,X_1,Y_0,\dots,Y_n]/r(X_i,Y_j)  \right)_{X_0}.
\end{equation}
Similarly one obtains
\begin{equation}\label{eq:Rlocalised1}
    \R_{X_1} = \left(\C[X_0,X_1,Y_0,\dots,Y_n]/r(X_i,Y_j)  \right)_{X_1}.
\end{equation}
On the other hand, we see
\begin{equation}\label{eq:Rirrelevant}
    \R / \langle X_0,X_1 \rangle = \C[Y_0,\dots,Y_n,W_1,\dots,W_l]/\langle g_{i_0}(Y_j), \dots, g_{i_l}(Y_j)\rangle.
\end{equation}
Let $T_1 = \Spec (\R / \langle X_0, X_1 \rangle)$ and $T_2 = \Spec (\R / \langle Y_0, \dots,Y_n \rangle)$. Since $r(X_i,Y_j)$ is irreducible, the rings $\R_{X_0}, \R_{X_1}$ are integral domains of dimension $n+2$ by \eqref{eq:Rlocalised0}, \eqref{eq:Rlocalised1} respectively. Additionally, since $\{g_{i_0},\dots,g_{i_l}\}$ form a regular sequence, we have $\dim T_1 = n$ by \eqref{eq:Rirrelevant}. One sees that $\dim (T_2 \cap \Spec \R_{X_0}) = 2$ via
\[
T_2 \cap \Spec \R_{X_0} = \Spec( \C[X_0,X_1,Y_0,\dots,Y_n]_{X_0}/\langle r(X_i,Y_j),Y_0,\dots,Y_j\rangle ) = \Spec \C[X_0,X_1]_{X_0}.
\]
In the same way, $\dim (T_2 \cap \Spec \R_{X_1}) = 2$. Since
\[
T_2 = (T_1 \cap T_2) \cup (\Spec \R_{X_0} \cap T_2) \cup (\Spec \R_{X_1} \cap T_2),
\]
we conclude that $\dim T_2 \leq \max \{\dim T_1,2 \} = n$. We may now consider the morphism
\[
\Spec \R \backslash (T_1 \cup T_2) \rightarrow Z,
\]
which is a $(\C^*)^2$-bundle. This implies that $\Spec \R \backslash (T_1 \cup T_2)$ is smooth and irreducible of dimension $n+2$. But $J \lhd U$ is generated by $l+1$ elements, and thus each irreducible component of $\Spec \R$ has dimension at least $n+2$. 
Thus $T_1 \cup T_2$ is not an irreducible component of $\Spec \R$, as $\dim T_1 \cup T_2 \leq n$.
Hence $T_1 \cup T_2$ is in the closure of $\Spec \R \backslash (T_1 \cup T_2)$, and 
we have that $\Spec \R$ is irreducible of dimension $n+2$. 
It follows that $J$ is a complete intersection ideal and hence $\Spec \R$ is Cohen-Macaulay.
On the other hand, $\Spec \R$ is smooth in codimension one, indeed it is smooth away from $T_1 \cup T_2$.
Hence $\Spec \R$ is normal by Serre criterion.
In particular, $\Spec \R$ is integral, and we are done.
\end{proof}

\begin{proposition}\label{prop:converse}
    Fix $n \geq 3$ and $d,e \geq 1$. Let $Z \subset X$ is non-singular, defined by a degree $(d,e)$ equation $r \in \R(X)$, admitting index sequence $\{i_0,\dots,i_l\}$. Let $U = \C[X_0, X_1, Y_0, \dots, Y_n, W_1,\dots,W_l]$ be the free $\Z^2$-graded polynomial ring with $\deg(X_i) = (1,0), \deg(Y_j) = (0,1)$ and $\deg(W_k) = (i_{k-1} - i_k,e)$ for the relevant $i,j,k$, and $J\lhd U$ the homogeneous ideal
    \begin{equation*}
       J = \langle X_1^{i_1 - i_0} W_1 + g_{i_0}(Y_j), X_1^{i_2 - i_1} W_2 + g_{i_1}(Y_j) - X_0^{i_1 - i_0} W_1,  \dots , g_{i_l}(Y_j) - X_0^{i_l - i_{l-1}} W_l\rangle.
    \end{equation*}
    If there is an isomorphism $\R(Z) \cong \R$ of $\Z^2$-graded algebras, where $\R = U/J$, then the sequence $g_{i_0},\dots,g_{i_l}$ is regular in $\C[y_0,\dots,y_n]$. In other words, $Z$ is regular with index sequence $\{i_0,\dots,i_l\}$.
\end{proposition}

\begin{proof}
    By degree considerations, the isomorphism $\R \cong \R(Z)$ sends the $\C$-span of $X_0, X_1$ to the $\C$-span of $x_0,x_1$. By \cite[Proposition 2.7]{MDS}, we conclude $X_0, X_1$ is a regular sequence in $\R$, and thus
    \[
    \dim \R / \langle X_0, X_1 \rangle = \dim \R -2 = \dim \R(Z) -2 = n,
    \]
    where again the final equality uses $Z$ being a Mori Dream Space. Using \eqref{eq:Rirrelevant} as found in the proof of the previous Theorem, we conclude that $\dim \C[Y_0,\dots,Y_n]/\langle g_{i_0}(Y_j), \dots, g_{i_l}(Y_j)\rangle = n-l$, implying that $g_{i_0},\dots,g_{i_l}$ is regular in $\C[y_0,\dots,y_n]$, as claimed.
\end{proof}

\section{Examples}

We look at some examples of our results of geometric interest. 

\begin{example} \label{ex:def1} Let $n=3$, $d=2$ and choose three general polynomials $g_0,g_1,g_2\in \C[y_0, y_1, y_2,y_3]$ of degree $e=4$, forming a regular sequence. Consider the family of varieties $q\colon {\mathcal Z} \to \A^1_t$ defined by
\[{\mathcal Z} = \left\{g_0 x_0^2 + t g_1 x_0x_1 + g_2 x_1^2 =0 \right\}\subset \p^1\times \p^3\times \A^1_t.
\]
For every $t\in\A^1$, the hypersurface fibre $Z_t\subset \p^1\times \p^3$ of the family $q\colon {\mathcal X} \to \A^1_t$ is a smooth Calabi--Yau threefold. Let $Z_t\stackrel{f_t}\longrightarrow {\bar Z}_t \stackrel{g_t}\longrightarrow \p^3$ be the Stein factorization of the second projection $p_2$. The map $g_t\colon {\bar Z}_t \to \p^3$ is a double cover in all cases, ramified over the divisor
\[ D_t = \{t^2g_1^2-4 g_0g_2=0\}\subset\p^3
\]
that is singular along the degeneracy locus 
\[Y_t=\{g_0=tg_1=g_2=0\}\subset\p^3.\]
For $t\neq 0$, $Y_t\subset\p^3$ is a set of $64$ points, and the map $f_t\colon Z_t\to {\bar Z}_t$ is a small contraction, contracting a set of $64$ rational curves to nodes. For $t=0$ on the other hand, the map $Z_0\to {\bar Z}_0$ is a divisorial contraction, contracting a divisor $E\subset Z_0$ to a genus-$33$ curve $C\subset {\bar Z}_0$ isomorphic 
to~$Y_0\subset \p^3$. As already observed by~\cite{constantinetal}, see in particular~\cite[3.3.2-3.3.3]{constantin}, for numbers of global sections, the Cox ring detects this change in birational behaviour. 
\begin{theorem}\label{thm_cy} For $t\neq 0$, the $\Z^2$-graded Cox ring $\R(Z_t)$ of the Calabi--Yau hypersurface $Z_t\subset \p^1\times \p^3$ can be presented as
\[ \R(Z_t)\cong k[x_0,x_1,y_0,y_1,y_2,y_3,z_1,z_2]/\langle x_1 z_1 + g_0, x_1 z_2 + g_1 - x_0 z_1, g_2 - x_0 z_2\rangle,\]
with variables of bidegrees $(1,0), (0,1)$ and $(-1,4)$ respectively. 
For $t=0$, we have 
\[\R(Z_0)\cong \C[x_0,x_1,y_0,y_1,y_2,y_3,w]/\langle x_1^2 w + g_0, g_2 - x_0^2 w\rangle,\]
with variables of bidegrees $(1,0), (0,1)$ and $(-2,4)$ respectively. 
In particular, every member of the family is a Mori dream space, with a complete intersection Cox ring, but the effective cone and Cox ring jump discontinuously in the family.  
\end{theorem}
\begin{proof} By the assumptions, for $t\neq 0$, the hypesurface $Z_t$ is regular with full index sequence $\{0,1,2\}$. On the other hand, the hypersurface $Z_0$ is regular, admitting index sequence $\{0,2\}$. Hence the statements follow from Theorem~\ref{thm:mainintro}. 
\end{proof}
Antonio Laface has informed us that the Cox rings in these examples can also be computed using the method of~\cite{lafaceetal}. 
\end{example}

We introduce one further, singular, member of this deformation family of varieties with interesting behaviour.

\begin{example} \label{ex:det} Choose general linear, respectively cubic polynomials $a_0,a_1,a_2\in R_1$ and 
$b_0,b_1,b_2\in R_3$. Consider the determinantal hypersurface
\[
Z = \left\{\begin{vmatrix}a_0 & a_1& a_2 \\ b_0 & b_1& b_2 \\ x_0^2 & x_0x_1 & x_1^2\end{vmatrix}=0\right\}
=\left\{g_0x_0^2  +g_1x_0x_1 +g_2 x_1^2 = 0\right\} \subset\p^1\times \p^3,\]
with $g_0=a_1b_2-a_2b_1, g_1=a_2b_0-a_0b_2, g_2=a_0b_1-a_1b_0$. 
Then the degeneracy locus is the curve
\[
Y = \left\{{\mathrm{rk}}\begin{pmatrix}a_0 & a_1& a_2 \\ b_0 & b_1& b_2 \end{pmatrix}\leq 1\right\}
 \subset \p^3,\]
a smooth space curve of genus $21$ and degree $13$. Its ideal $I_Y\lhd R$ has a resolution~\eqref{eq:sequence_for_T} in Hilbert--Burch form
 \begin{equation}0 \longrightarrow R^{2}\stackrel{B}\longrightarrow R^{3}\stackrel{A}\longrightarrow R \longrightarrow R/I_Y\longrightarrow 0.\end{equation}
Here, $B=\begin{pmatrix}a_0 & a_1& a_2 \\ b_0 & b_1& b_2 \end{pmatrix}^t$,  
and $A=\bigwedge^2 B = (g_0, g_1, g_2)^t$. As a determinantal hypersurface, $Z$ is singular along the locus given by the $2\times 2$ minors of its defining matrix, which gives the locus 
\[
\mathop\mathrm{Sing} Z=\left\{g_0=g_1=g_2=a_1^2-a_0a_2=x_0a_1-x_1a_0=x_0a_2-x_1a_1=0\right\}\subset \p^1\times \p^3,
\]
a set of~$26$ isolated ordinary double points all lying on the ruled surface $E\subset Z$. Blowing up this ruled surface, a Weil divisor through each of the ODP's, gives a small resolution $\widetilde Z\rightarrow Z$, a smooth Calabi--Yau model. 

Since $Z\subset X=\p^1\times\p^3$ has isolated nodal singularities, the restriction map $\pic(X)\to\pic(Z)$ is still an isomorphism. 
However, $Z$ is not $\Q$-Cartier, so the ring $\R(Z)$ as defined above only contains sections of Cartier divisors.

By Proposition~\ref{prop:ker}(ii), the columns of $B$ give us elements $f_1\in N_{(-2,5)}$ and $f_2\in N_{(-2,7)}$ that together generate the $R$-module $\bigoplus_{b \in \Z} N_{(-2,b)}$. Let $u=\epsilon f_1 \in \R(Z)_{(-2,5)}$ and $w =\epsilon f_2\in \R(Z)_{(-2,7)}$ be the lifts of these module generators to elements of the Cox ring as in Proposition~\ref{prop:epsilon}. By Proposition~\ref{xmult}, these satisfy equations
\begin{align*}
    x_0 u = a_1 z_1 + a_2 z_2, \quad x_1 u = a_2 z_1 + a_3 z_2, \\
    x_0 w = b_1 z_1 + b_2 z_2, \quad x_1 w = b_2 z_1 + b_3 z_2.
\end{align*}
Furthermore, by comparing $\delta(u), \delta(w),\delta(z_1^2),\delta(z_1 z_2),\delta(z_2^2)$ we have
\begin{equation*}
    z_1^2 + b_2 u - a_2 w = z_1 z_2 - b_1 u + a_1 w = z_2^2 + b_0 u - a_0 w = 0.
\end{equation*}
Consider the free $\Z^2$-graded $k$-algebra \[
    T = \C[X_0,X_1,Y_0, \dots,Y_n,Z_1,Z_2,U,W]
    \]
    with generators in degrees $(1,0), (0,1), (-1,4), (-2,5)$ and $(-2,7)$ respectively. Let $I \lhd T$ be the ideal 
    \begin{gather*}
       I= \langle X_1 Z_1 + g_0(Y_i), \quad X_1 Z_2 + g_1(Y_i) - X_0 Z_1, \quad g_2(Y_i) - X_0 Z_2,\\
        X_0 U - a_1(Y_i) Z_1 - a_2(Y_i) Z_2, \quad X_1 U - a_2(Y_i) Z_1 - a_3(Y_i) Z_2, \\
    X_0 W - b_1(Y_i) Z_1 + b_2(Y_i) Z_2, \quad X_1 W - b_2(Y_i) Z_1 - b_3(Y_i) Z_2,\\
    Z_1^2 + b_2(Y_i) U - a_2(Y_i) W, \quad Z_1 Z_2 - b_1(Y_i) U + a_1(Y_i) W, \quad Z_2^2 + b_0(Y_i) U - a_0(Y_i) W\rangle.
    \end{gather*}
    Then our observations so far prove the existence of an algebra homomorphism 
    \[
    \phi: T/I \rightarrow \R(Z),
    \]
 to the (Cartier) Cox ring  $\R(Z)$ of $Z$ that induces isomorphisms $(T/I)_{(a,b)} \cong \R(Z)_{(a,b)}$ whenever $a \geq -2$.

We sketch an argument provided to us by Antonio Laface that proves that for a general determinantal hypersurface~$Z$, the map $\phi$ 
is an isomorphism. By~\cite[Thm.1]{lafaceetal}, the Cox ring $\R(Z)$ is the intersection of certain localizations of quotients
of~$\R(\p^1\times\p^3)$. This intersection can be computed using the ideas of~\cite[Cor.2.4]{lafaceetal}, which gives the result that~$\phi$ is surjective under certain dimension and saturation conditions. The latter can be checked for general $a_i,b_j$ using computer algebra.
\end{example} 

\begin{example}\label{ex:linear}
    Consider the hypersurface 
    \[Z=\{
    y_0 x_0^4 + (-\lambda_0 y_0 -\lambda_1 y_1 -\lambda_2 y_2 -\lambda_3 y_3)x_0^3 x_1 + y_1 x_0^2 x_1^2 + y_2 x_0 x_1^3 + y_3 x_1^4 = 0\}\subset \p^1 \times \p^3
    \]
    of bidegree $(d,e) = (4,1)$ with non-zero scalars $\lambda_j\in\C$, noting that a general degree $(4,1)$ hypersurface in $\p^1 \times \p^3$ has this form after a coordinate transformation.
    As the five linear forms are linearly dependent, by Theorem~\ref{thm:mainintro}(ii), the Cox ring is not of the form~\eqref{presentationR}, contrary to the statement \cite[Thm.1.1(iv)]{ottem}. We investigate the Cox ring of this particular family in this example. 


The kernel of $A_2$ as defined in Proposition~\ref{prop:ker} contains $(\lambda_0,1,\lambda_1,\lambda_2,\lambda_3)$, defining a section $w \in \R(Z)_{(-2,1)}$,
which satisfies the equations
\begin{equation}\label{eq:linear}
       x_0 w = \lambda_0 z_1 + z_2 +\lambda_1 z_3 +\lambda_2 z_4, \quad x_1 w = z_1 + \lambda_1 z_2+ \lambda_2 z_3 +\lambda_3 z_4  .
\end{equation}
    These equations
    imply that for general choices of $\lambda_j$, 
    the subalgebra $\R'$ of $\R(Z)$ generated by $x_0, x_1,w,z_3,z_4$ contains $z_1, z_2$.
Note also that the relations
\[
x_1 z_1 + y_0=0, \quad x_1 z_3 + y_1 - x_0 z_2=0, \quad x_1 z_4 + y_2 - x_0 z_3=0,\quad y_3 - x_0 z_4=0
\]
hold in $\R(Z)$.  
Hence the subalgebra $\R'$ contains $y_0, y_1,y_2,y_3$ as well.
In fact, we can show that $\R(Z)=\R'$, which is a polynomial ring as follows.

\begin{proposition}\label{prop:linear}
         For general constants $\lambda_j\in\C$, the $\Z^2$-graded Cox ring $\R(Z)$ of the hypersurface $Z \subset \p^1 \times \p^3$ is isomorphic to the polynomial ring $\C[x_0,x_1,w,z_3,z_4]$.    
\end{proposition}

\begin{proof}
Let $r \in H^0(\Ox_{\p^1 \times \p^3}(4,1)) =H^0(\Ox_{\p^1}(4)) \otimes H^0( \Ox_{\p^3}(1)) $ be the section which defines $Z$.
Then $r$ induces an exact sequence
\begin{align*}
0 \to \Ox_{\p^1}(-4) \xrightarrow{r} H^0( \Ox_{\p^3}(1)) \otimes \Ox_{\p^1} \to \mathcal{E} \to 0
\end{align*}
on $\p^1$
and we have $Z =\p_{\p^1}(\mathcal{E})\subset \p^1 \times \p^3$.
We note that $\mathcal{E}$ is locally free of rank $3$
since $r$ is general.
Let $\mathcal{E} \simeq \Ox_{\p^1} (a) \oplus \Ox_{\p^1} (b) \oplus \Ox_{\p^1} (c)$ with $a+b+c=4$.
Since $\mathcal{E}$ is a quotient of a trivial bundle, we have $a,b,c \geq 0$.
If $a =0$, then $r$ must be contained in $H^0(\Ox_{\p^1}(4)) \otimes V'$ for
$V' =\ker H^0(\Ox_{\p^3}(1)) \to H^0(\Ox_{\p^1}(a))=\C$,
which contradicts the generality of $r$.
Hence $a,b,c \geq 1$ and $\mathcal{E} \simeq  \Ox_{\p^1} (1) \oplus \Ox_{\p^1} (1) \oplus \Ox_{\p^1} (2)$.

Since the embedding $Z =\p_{\p^1}(\mathcal{E})\subset \p^1 \times \p^3$ over $\p^1$ is induced by 
$H^0( \Ox_{\p^3}(1)) \otimes \Ox_{\p^1} \twoheadrightarrow \mathcal{E} $,
the tautological line bundle $\Ox_{\mathcal{E}}(1)$ on $ \p_{\p^1}(\mathcal{E})$ is identified with $\Ox_Z(0,1)$,
where $\Ox_Z(a,b) \coloneqq \Ox_{\p^1\times \p^3}(a,b)|_Z $.
Since $\mathcal{E}(-1) \simeq 
\Ox_{\p^1} \oplus \Ox_{\p^1} \oplus \Ox_{\p^1} (1)$, 
the morphism 
\[\mu : Z=\p_{\p^1}(\mathcal{E}(-1)) \to P = \p H^0(\Ox_Z(-1,1)) \cong\p^3\]
induced by $|\Ox_Z(-1,1)|$ is the blow-up along a line $l$.
We note that $P$ is different from the second factor of $\p^1 \times \p^3$.

Let 
\begin{align*}
D_i =(x_i=0), \quad D'_k=(z_k =0), \quad E=(w=0) 
\end{align*}
be divisors on $Z$.
Since $w$ is a non-zero section of $H^0( \Ox_Z(-2,1)) = H^0(\mathcal{E}(-2)) \simeq \C$,
$E$ is the exceptional divisor of the blow-up $\mu$.

Since $z_1,z_2,z_3,z_4$ are linearly independent over $R$ by Proposition \ref{prop:z_k},
so are over $\C$.
Hence $z_1,z_2,z_3,z_4$ form a basis of $ H^0(\Ox_Z(-1,1))  =H^0(\mathcal{E}(-1)) \simeq \C^4$.
By (\ref{eq:linear}), $x_0w,x_1w,z_3,z_4$ also form a basis of $ H^0(\Ox_Z(-1,1)) =H^0(P, \Ox(1))$. We regard $P=\p^3$ as a toric variety by this basis.
Let $H_i=(x_iw=0) \subset P$ be the torus invariant planes for $i=0,1$.
Since $\mu^*H_i=D_i+E$, the line $l =\mu(E)$ is $H_0 \cap H_1$, which is torus invariant.
Hence $\mu$ is a toric morphism and $Z$ is a toric variety as well.
Then torus invariant prime divisors of $Z$ are $D_0,D_1,E, D'_3,D'_4$.
Since the Cox ring of a smooth projective toric variety is a polynomial ring whose variables correspond to torus invariants prime divisors,
we conclude that $\R(Z) \simeq \C[x_0,x_1,w,z_3,z_4]$.
\end{proof}

    
\end{example}

As a last example, we discuss from our point of view a phenomenon in the birational geometry of the hypersurface $Z$ that formed a key part of the argument of~\cite{ottem}.

\begin{example}\label{ex:regcasemodel}
    Let $n \geq 3$ and $d,e \geq 1$. Choose general polynomials $g_0,\dots,g_d \in \C[y_0,\dots,y_n]$ of degree $e$, forming a regular sequence. Then
    \[
    \R(Z) \cong \C[x_0,x_1,y_0,\dots,y_n,z_1,\dots,z_d]/\langle x_1 z_1 + g_0, x_1 z_2 + g_1 - x_0 z_1,\ldots, g_d - x_0 z_d\rangle.
    \]
    Let $S_{Z^+}$ be the subalgebra of $\R(Z)$ generated by $y_0,\dots,y_n,z_1,\dots,z_d$. By Proposition~\ref{prop:generalWideal}, proved in the final section, we have an isomorphism
    \[
    S_{Z^+} \cong \C[y_0,\dots,y_n,z_1\dots,z_d]/J',
    \]
    where $J'$ is the ideal generated by the equations
    \begin{equation} \label{eq:secondmodel}
     g_{k} (z_{k'+1} z_{k''} - z_{k'} z_{k''+1}) - g_{k'} (z_{k+1} z_{k''} - z_k z_{k''+1}) + g_{k''} (z_{k+1} z_{k'} - z_k z_{k'+1}),
    \end{equation}
    for $0 \leq k < k' < k'' \leq d$. The ring $S_{Z^+}$, defined by the equations \eqref{eq:secondmodel}, is the homogeneous coordinate ring of a birational model $Z^+ \subset \p^d_{z_k} \times \p^n_{y_j}$ of $Z$; compare~\cite[Section 2]{ottem}.
\end{example}

\section{A problem in algebra}


Fix $d \geq 0$ and fix an admissible index sequence $\{i_0,\ldots, i_l\}$ for some $0 \leq l \leq d$; note that $d = 0$ is equivalent to $l = 0$ by our conventions. Let $S = \C[Y_{i_0}, \dots, Y_{i_l}]$ be the free commutative $\C$-algebra on $l+1$ generators. For each $a \geqslant 1$, define a map of free $S$-modules \[A_a : S^{a+d-1} \rightarrow S^{a-1}\] by the $(a-1) \times (a+d-1)$ matrix with $(s,t)$ entry equal to $Y_{i_k}$ if $t-s = i_k$ and $0$ otherwise. Denote $K_a=\ker A_a$ and let $K = \bigoplus_{a\geq 1} K_a$, which we regard as a graded $S$-module. Note that $A_1=0$ and so $K_1\cong S^d$, with the convention that $S^0 = 0$.

\begin{example}
    If $d = 5, l = 2$ and the degree sequence is $\{0,3,5\}$, then $A_3$ is given by
\begin{equation*}
    \begin{pmatrix}
   Y_{0} & 0 & 0 & Y_{3} & 0 & Y_{5} & 0 \\
   0 & Y_{0} & 0 & 0 & Y_{3} & 0 & Y_{5}
    \end{pmatrix}.
    \end{equation*}
\end{example}

We will mostly employ formal algebraic ideas in this section, but sometimes the use of the following auxiliary hypersurface will be helpful. Consider the non-singular hypersurface
    \[
    Z' = \{Y_{i_0} x_0^{d} + Y_{i_1} x_0^{d-i_1} x_1^{i_1} + \dots + Y_{i_l} x_1^d\} \subset \p_{x_i}^1 \times \p_{Y_{i_k}}^l.
    \]
We can repeat the arguments of Section 2 for this hypersurface. Its Cox ring contains sections $w_1,\dots,w_l\in\R(Z')$ as given in~\eqref{eq:wkdef}, and as in equations~\eqref{eq:zinxw}, contains the sections $z_1,\dots,z_d\in\R(Z')$ given in Proposition~\ref{prop:coxz1}. We also have an identification 
\begin{equation}\label{eq:KandZ'} K\cong \bigoplus_{-a \leq -1} \bigoplus_{b\in\Z} \R(Z')_{(-a,b)}\end{equation}
as graded $S$-modules via Proposition~\ref{prop:ker}.

\begin{definition}
    Fix $a \geq 1$.
\begin{itemize}
    \item[(i)] Define the operators $\tilde x_0, \tilde x_1\colon K_{a+1} \rightarrow K_a$ respectively as truncation by one entry on the right, respectively on the left.
    \item[(ii)] For each $1 \leq k \leq d$ we define $\tilde z_k: K_a \rightarrow K_{a+1}$ as follows. Suppose that $(f_m)_{m=1}^{a+d-1} \in \ker A_a$. Then $ \tilde z_k (f_m)_{m=1}^{a+d-1}$ is given by $(f'_m)_{m=1}^{a+d}$, where
    \begin{align*}
        f'_m =& f_m Y_k + \dots + f_{m+d-l} Y_d,& \quad m = 1,\dots,a+k-1, \\
        f'_m =& -f_{m-k} Y_0 - \dots - f_{m-1} Y_{k-1},& \quad m = k+1,\dots, a+d,
    \end{align*}
    where $Y$ variables of undefined index are set to be zero.
    \item[(iii)] For each $1 \leq k \leq l$, we define $\tilde w_k : K_{a} \rightarrow K_{a+i_k - i_{k-1}}$ by stipulating that
    \[
    \tilde x_0^{m-1} \tilde x_1^{i_k - i_{k-1}-m} \tilde w_k = \tilde z_{i_{k-1}+m}
    \]
    for each $1 \leq m \le i_k - i_{k-1}$.
\end{itemize}
    \label{operators1}
\end{definition}

\begin{proposition} \label{prop:opwelldefined}
The operators $\tilde x_i, \tilde z_k, \tilde w_k\colon K\to K$ in Definition~\ref{operators1} are well defined. 
\end{proposition}
\begin{proof} For $\tilde x_i$, the statement is clear. For $z_k$, when $a = 1$, then each $f'_m$ is specified by one formula. If $a > 1$, then for $m = k+1, \dots, a+k-1$, the two given definitions of $f'_m$ agree, since $(f_m)_{m=1}^{a+d-1} \in \ker A_a$. Similarly, it is easy to see directly that the operators $\tilde w_k$ are well defined. 

An alternative argument employs the auxiliary hypersurface $Z'$ defined above. 
The operators in Definition~\ref{operators1} correspond to multiplication by the corresponding sections $x_i, z_k, w_k\in \R(Z')$ by Propositions~\ref{xmult}--\ref{zmult} under the identification~\eqref{eq:KandZ'}, and are thus well defined.  
\end{proof}

Let $U = S[X_0, X_1, W_1,\dots,W_l]$ be the free graded $S$-algebra on $l+2$ generators, with $\deg(X_i)=1$ and $\deg(W_k) = i_{k-1}-i_k$ for relevant $i,k$. For a general graded $S$-module $U'$, we denote by $U'_{\leq -1}$ the graded submodule $\bigoplus_{a\leq -1} U'_{a}$. It is also useful to define the elements
\begin{equation} \label{eqn:ZinXWalg}
Z_{i_{k-1}+m} = X_0^{m-1}X_1^{i_k - i_{k-1}-m}W_k \in U_{-1},
\end{equation}
for $1\leq k \leq l$ and $1 \leq m \leq i_k - i_{k-1}$.

\begin{proposition} \label{prop:psialg}
    There is a map of $S$-modules
    \[
    \psi: U_{\leq -1} \rightarrow K
    \]
    defined as follows. Define, for each $1 \leq k \leq l$, the element $\psi(W_k) \in K_{i_k - i_{k-1}}$ to be $(f_m)_{m=1}^{i_k - i_{k-1} + d -1}$ with $f_m = 1$ if $m = i_k$ and all other $f_m$ equal 0. The grading on $U$ implies that any monomial $M$ in $U_{\leq -1}$ divides by $W_k$ for some $k$. If $M = X_0^{\alpha_0} X_1^{\alpha_1} W_1^{\beta_1}\dots W_l^{\beta_l}$ with say $\beta_k >0$, then we set
    \[
    \psi(M) = \tilde x_0^{\alpha_0} \tilde x_1^{\alpha_1} \tilde w_1^{\beta_1}\dots \tilde w_k^{\beta_k -1} \dots \tilde w_l^{\beta_l} \psi(W_k).
    \]
    Furthermore, the elements $\psi(Z_k) \in K_1$, for $1 \leq k \leq d$, give the standard basis of $K_1 \cong S^d$.
\end{proposition}

\begin{proof}
    To argue that $\psi$ is well defined, we need only argue the operators $\tilde x_i$ and $\tilde w_k$ all commute with each other, as well as establish that $\tilde w_k \psi(W_{k'}) = \tilde w_{k'} \psi(W_k)$ for each $1 \leq k, k' \leq l$. 

    It is possible to establish this claim directly algebraically, using the definitions of the operators. 
    This somewhat long-winded argument will feature in \cite{pollock}.

    To give a short argument instead, notice as above that the operators correspond to multiplication of sections in the Cox ring $\R(Z')$ of the auxiliary hypersurface $Z'$. Since multiplication of these sections is clearly commutative, all of the operators must commute with each other, and the extra property
    $\tilde w_k \psi(W_{k'}) = \tilde w_{k'} \psi(W_k)$ for each $1 \leq k, k' \leq l$ also holds.

    The final statement is easily checked.
\end{proof}

\begin{theorem}\label{thm:generalgebra}
    Let $J \lhd U$ be the homogeneous ideal defined by
    \begin{equation*}
       J = \langle X_1^{i_1 - i_0} W_1 +Y_{i_0}, X_1^{i_2 - i_1} W_2 + Y_{i_1} - X_0^{i_1 - i_0} W_1,  \dots , Y_{i_l} - X_0^{i_l - i_{l-1}} W_l\rangle. 
   \end{equation*}
    Then $\psi:U_{\leq -1} \rightarrow K$ is surjective with kernel $\ker \psi = J_{\leq -1}$. In particular there is an isomorphism of $S$-modules $K \cong (U/J)_{ \leq -1}$.
\end{theorem}

We prove Theorem~\ref{thm:generalgebra} using induction in $l$. If $l = 0$ then $d = 0$, the result is true since in this case $U_{ \leq -1} = K = 0$. We thus assume in the following that $l \geq 1$. For the induction step we need some Definitions and Lemmas.

\begin{definition}\label{def:inductionbar} Let $\bar S = \C[Y_{i_0},\dots,Y_{i_{l-1}}]$. For each $a \geq 1$, let 
\[\bar A_a : \bar S^{a+d'-1} \rightarrow \bar S^{a-1}\]
be the $\C$-module maps as defined at the start of this section, but for the case of $l$ variables with partition $0 = i_0 < \dots < i_{l-1} = d'$, with corresponding auxiliary hypersurface. Denote $\bar K_a=\ker \bar A_a$ and $\bar K = \bigoplus_{a \geq 1} \bar K_a$. Define the operators $\tilde \zeta_k\colon\bar K_a\to \bar K_{a+i_k - i_{k-1}}$ for $1 \leq k \leq l-1$ analogously to the $\tilde w_k$ in Definition~\ref{operators1}(iii). We abuse notation slighlty and still use $\tilde x_0$ and $\tilde x_1$ for the respective truncation operators. Define $\bar U = S[X_0,X_1,\bar \zeta_1,\dots,\bar \zeta_{l-1}]$, and define the ideal $\bar J \lhd \bar U$ and map $\bar \psi : \bar U_{\leq -1} \rightarrow \bar K$ analogously to $J$ and $\psi$ in Proposition~\ref{prop:psialg}.
\end{definition}

\begin{lemma}\label{lem:inductionw}
    Consider an element $\psi (W_{k_s} \cdots W_{k_1} W_{k_0}) \in K_t$, where no $k_i$ is equal $l$. Setting $Y_{i_{l}} = 0$ in the resulting element of $S^{t+d-1}$ and discarding the final $i_l - i_{l-1}$ components gives us an element of $\bar S^{t+d'-1}$. This coincides with $\psi(\bar \zeta_{k_s} \cdots \bar \zeta_{k_1} \bar \zeta_{k_0}) \in \bar K_{t}$.
\end{lemma}

\begin{proof}
    This follows immediately from the definitions.
\end{proof}

\begin{lemma}\label{lem:inddet}
    Suppose $u = (f_m)_{m = 1}^{a + d-1} \in K_a$. If $a > i_l - i_{l-1}$ then $u$ is determined by its first $a + i_{l-1} - 1$ entries. If $a \leq i_l - i_{l-1}$ then $u$ is determined by its first $a + i_{l-1} - 1$ entries up to free choices of the entries $f_{a+ i_{l-1}},\dots, f_{i_l}$.
\end{lemma}

\begin{proof}
    If $a > i_l - i_{l-1}$, then the first $i_l - i_{l-1}$ rows of $A_a$ give equations determining the remaining entries. If $a \leq i_l - i_{l-1}$, then the rows of $A_a$ determine each of the last $a-1$ entries in $u$ and the $a+i_{l-1}, \dots, i_l$ columns of $A_a$ are zero columns, giving the remaining free choices.
\end{proof}

\begin{proof}[Proof of Theorem~\ref{thm:generalgebra}]
    We know the result is true for $l = 0$, so suppose that $1 \leq l \leq d$. We first prove that $\psi$ is surjective. Suppose that $a \geq 1$ and $u = (f_m)_{m=1}^{a+d-1}\in K_a$ is given. Setting $Y_{i_l} = 0$ and discarding the last $i_l - i_{l-1}$ elements we obtain $\bar u = (\bar f_m)_{m=1}^{a+d'-1} \in \bar K_a$. By induction in $l$ we can find $p(X_0,X_1,\bar \zeta_1,\dots,\bar \zeta_{l-1}) \in \bar U_{\leq -1}$ such that $\bar \psi(p) = \bar u$. By Lemma~\ref{lem:inductionw} we have
    \[
    u' = u - \psi \left(p(X_0,X_1,W_1,\dots,W_{l-1})\right) = (f'_1 Y_{i_l},\dots,f'_{a+d'-1} Y_{i_l},F_{a+d'},\dots,F_{a+d-1}) \in K_a
    \]
    for some $f'_1,\dots,f'_{a+d'-1},F_{a+d'},\dots,F_{a+d-1} \in S$. We continue inductively in $a$.
    
    Suppose $1 \leq a \leq i_l - i_{l-1} = d-d'$. Let $v = (f'_1,\dots,f_{a+d'-1},0,\dots,0) \in K_1$, where the number of zeroes at the end is $d-d' -a +1$. By the final statement of Proposition~\ref{prop:psialg}, 
    \[v = \psi(f'_1 Z_1 + \dots + f'_{a+d'-1}Z_{a+d'-1})
    \]
    is in the image of $\psi$. Now consider $\tilde x_0^{d-d' -a +1} \tilde w_l (v) \in K_a$, which is equal
    \[
    (f'_1 Y_d,\dots,f'_{a+d'-1}Y_d,0,\dots,0,f''_{d+1},\dots,f''_{a+d-1}).
    \]
    Now $\tilde x_0^{d-d' -a +1} \tilde w_l (v)$, which is in the image of $\psi$ since $v$ is, and $u'$ have the same first $a+d'-1$ entries in common, thus the last $a-1$ entries of each must also agree by Lemma~\ref{lem:inddet}. Then
    \[
    u' - \tilde x_0^{d-d' -a +1} \tilde w_d (v) = (0,\dots,0,F_{a+d'},\dots,F_{d},0,\dots,0) \in K_a.
    \]
    But this is in turn equal to
    \[
    \psi\left((F_{a+d'} X_1^{d-d'-a} + \dots + F_{d} X_0^{d - d' -a})W_l\right),
    \]
    so that $u$ can be written in the image of $\psi$.

    If instead $a > d - d'$, then $(f'_1 Y_{i_l},\dots,f'_{a+d'-1} Y_{i_l}) \in K_{a-d+d'}$. Since $Y_{i_l}$ is not a zero divisor in $S$, we have $v = (f'_1,\dots,f'_{a+d'-1}) \in K_{a-d+d'}$. Since the first $a+d'-1$ coordinates of $u'$ and $\tilde w_l(v)$ are equal, we have $u' = \tilde w_l (v)$ by Lemma~\ref{lem:inddet}. By induction in $a$, $v$ lies in the image of $\psi$, and thus so does $u$.

    We now turn to the kernel of $\psi$. We can check that $J_{\leq -1} \subset \ker \psi$. Suppose that $p \in (\ker \psi)_{-a}$ for some $a \geq 1$. Consider $\bar p \in \bar U$ obtained from $p$ by setting $Y_{i_l} = W_l = 0$ and replacing each $W_k$ with $\bar \zeta_k$ for each $1 \leq k \leq l-1$. By Lemma~\ref{lem:inductionw}, $\bar p \in \ker \bar \psi$. By induction in $l$, we have $\bar p \in \bar J$, and thus $\bar p$ may be written as an $S$-linear combination of the equations
\[
X_1^{i_{k+1} - i_{k}} \bar \zeta_{k+1} + Y_{i_k} - X_0^{i_k - i_{k-1}} \bar \zeta_{k}, \quad 0 \leq k \leq l-1,
\]
where undefined indices are set to be zero. Write $\bar p$ as a $\bar U$-linear combination of these equations, except leaving the variable $\bar \zeta_l$ in where it is usually set it equal to zero. Let $p' \in U$ be the corresponding element of $J$ obtained from this linear combination by replacing $\bar \zeta_k$ with $W_k$ for each $1 \leq k \leq l$, so that $p' \in J_{\leq -1}$. We can replace $p$ with $p-p'$ without affecting whether it belongs to $J_{\leq -1}$. By construction, all terms of $p$ now divide by $Y_{i_l}$ or $W_l$. Adding a multiple of $Y_{i_l} - X_0^{i_l - i_{l-1}}W_l$, which is in $J$, we can replace $p$ with something dividing by $W_l$. Write $p = q W_l$.

If $a \leq i_l - i_{l-1}$ then $q \in U_{\geq 0}$. Consider the section \[q(x_0,x_1,w_1,\dots,w_l)\in\R(Z'),\] in the Cox ring of the auxiliary hypersurface $Z'$. The section $q(x_0,x_1,w_1,\dots,w_l)w_l \in \R(Z')$ is identified with $\psi(p) = 0$ in \eqref{eq:KandZ'} and is thus the zero section. Since $Z'$ is irreducible and normal, its Cox ring has no zero divisors, thus $q(x_0,x_1,w_1,\dots,w_l) = 0$ in $\R(Z')$. By Corollary~\ref{cor:partialdescription}, $q(X_0, X_1,W_1,\dots,W_l) \in \psi(J)$ . Then $p \in J_{\leq -1}$.

If $a > i_l -i_l$, then $0 = \psi (p) = \tilde w_l \psi(q)$. Now $\tilde w_l$ is an injective operator by Lemma~\ref{lem:inddet} and the fact it is multiplication by $Y_{i_l}$ on the first $a + d'-1$ terms. Thus $\psi(q) = 0$. By induction, $q \in J_{\leq -1}$ and thus $p\in J_{\leq -1}$. 
\end{proof}

We add a final Proposition, whose generalisation in Proposition~\ref{prop:generalWideal} is used for Example~\ref{ex:regcasemodel}.

\begin{proposition}\label{prop:universalWideal}
    Suppose that $l = d$ and $J, U$ are as in Theorem~\ref{thm:generalgebra}. Let $U' = S[W_1,\dots,W_d]$ be the subalgebra of $U$ generated by $W_1,\dots,W_d$. Let $J' \lhd U'$ be the homogeneous ideal generated by the equations
    \[
    Y_k (W_{k'+1} W_{k''} - W_{k'} W_{k''+1}) - Y_{k'} (W_{k+1} W_{k''} - W_k W_{k''+1}) + Y_{k''} (W_{k+1} W_{k'} - W_k W_{k'+1})
    \]
    for $0 \leq k < k' < k'' \leq d$. Then the natural map $\phi:U'_{\leq -1} \rightarrow (U/J)_{\leq -1} $ is surjective and has kernel $J'$.
\end{proposition}

\begin{proof}
    The equations correspond exactly to the images of the standard basis vectors in the degree $3$ map of the Koszul complex of the sequence $Y_0,Y_1,\dots,Y_d$ in $S$, and thus are contained in $J$. Alternatively, that $J' \subset J$ can be checked explicitly by writing the generators of $J'$ explicitly as $S$-combinations of the generators of $J$.
    
    We use induction in $d$. If $d = 0$ the result is immediate, so suppose that $d \geq 1$. Let $\bar U'$ and $\bar J'$ be the respective analogues of $U'$ and $J'$ with parameters $d-1 = l-1$. The equations in $J$ imply that any element in $(U/J)_{\leq -1}$ can be written as a polynomial in the $Y_j, W_k$, thus $\phi$ is surjective. Notice next that $(U/J)_{-1} \cong S^d \cong (U'/J')_{-1}$. Suppose that $p(Y_j,W_k) \in \ker \phi$ has degree $-a \leq -2$. Then $p(Y_j,W_k)$ is also in $\ker \psi = J$. As in the proof of Theorem~\ref{thm:generalgebra}, we can assume all terms of $p$ divide by $Y_d, W_d$. The equations of $J'$ can then be used to replace $p$ with a polynomial dividing by $W_d$, so that $p = qW_d$. Since $\tilde w_d$ is an injective operator, we conclude that $q \in \ker \phi$. Using induction in the degree~$a$, we have $q\in J'$ and thus $p \in J'$.
\end{proof}

\begin{theorem} \label{regularsequencekernels}
    Fix $d \geq 0$ and fix an admissible index sequence $\{i_0,\ldots, i_l\}$ for some $0 \leq l \leq d$. Let $S$ be a finitely generated graded $\C$-algebra with projective dimension $\kappa \geq l$. Suppose that $g_{i_0},\dots,g_{i_l}$ are homogeneous elements of degree $e$ forming a regular sequence in $S$. For each $a \geq 1$, define a map of free $S$-modules $A_a : S^{a+d-1} \rightarrow S^{a-1}$ by the $(a-1) \times (a+d-1)$ matrix with $(s,t)$ entry equal to $g_{i_k}$ if $t-s = i_k$ and $0$ otherwise. Let $K_a = \ker A_a$ and $K = \bigoplus_{a\geq 1} K_a$. Let $U = S[X_0,X_1,W_1,\dots,W_l]$ be the free $\Z^2$-graded $S$-algebra with $\deg_U(s) = (0,\deg_S(s)), \deg(X_i) = (1,0)$ and $\deg(W_k) = (i_{k-1}-i_k,e)$ for $s \in S$ and relevant $i,k$.     Then there is a surjective map
    \[
    \psi : U_{\leq -1} = \bigoplus_{-a \leq -1} \bigoplus_{b\in \Z} U_{(-a,b)} \rightarrow K,
    \]
whose kernel contains $J_{\leq -1}$, where $J \lhd U$ is the homogeneous ideal given by
    \[
       J = \langle X_1^{i_1 - i_0} W_1 + g_{i_0}(Y_j), X_1^{i_2 - i_1} W_2 + g_{i_1}(Y_j) - X_0^{i_1 - i_0} W_1,  \dots , g_{i_l}(Y_j) - X_0^{i_l - i_{l-1}} W_l\rangle.
    \]
\end{theorem}

\begin{proof}
    If $S = \C[Y_{i_0},\dots,Y_{i_1}]$ and $g_{i_k} = Y_{i_k}$ for each $k = 0,\dots, l$, then the statement is that of Theorem~\ref{thm:generalgebra} with refined degrees. All of the proofs given above can be modified to this more general setting. The operators in Definition~\ref{operators1} are defined in the same way with the $Y_{i_k}$ replaced by the $g_{i_k}$. Proposition~\ref{prop:psialg} holds as we symbolically only replace the $Y_{i_k}$ with $g_{i_k}$, preserving commutativity. In Definition~\ref{def:inductionbar}, we set $\bar S = S/(g_{i_l})$, which has projective dimension $\kappa -1 \geq l-1$. For Lemma \ref{lem:inductionw}, rather than setting $Y_{i_l} = 0$, we reduce modulo $g_{i_l}$. Lemma~\ref{lem:inddet} holds with regularity of the sequence $g_{i_0},\dots,g_{i_l}$ in $S$.

    Now the proof follows that of Theorem~\ref{thm:generalgebra}. We use induction in $(\kappa,l,a)$. If $\kappa = 0$ then $l = 0$, so $d = 0$ and the result is trivially true. If $\kappa \geq 1$ and $l =0$, again the result is trivially true. If $l \geq 1$ and $u \in K_a$, then reducing $u$ modulo $g_d$ and discarding the last $i_l - i_{l-1}$ terms puts us in the setting with parameters $(\kappa -1, l-1,a)$, where we may use induction. This gives us the surjectivity of $\psi$ and one checks that $J_{\leq -1} \subset \ker \psi$.
\end{proof}

\begin{remark}
    One can prove that, in the setting of Theorem~\ref{regularsequencekernels}, the kernel of $\psi$ is precisely $J_{\leq -1}$ in the cases $l=d$, or with an open condition on the $g_{i_k}$, purely algebraically from the operators defined. Since the argument used in Theorem~\ref{thm:general} is stronger, proving that $\ker \psi = J_{\leq -1}$ when the auxiliary hypersurface
    \[
    Z' = \{g_{i_0} x_0^d + g_{i_1} x_0^{d-i_1} x_1^{i_1} + \dots + g_{i_l} x_1^d \} \subset \p^1_{x_i} \times \Proj S
    \]
    is non-singular, we omit this.
\end{remark}

\begin{proposition}\label{prop:generalWideal}
    Suppose that $l=d$ and $J,U$ are as in Theorem~\ref{regularsequencekernels}. Let $U' = S[W_1,\dots,W_d]$ be the subalgebra of $U$ generated by $W_1,\dots,W_d$. Let $J' \lhd U'$ be the homogeneous ideal generated by the equations
    \[
    g_{k}(Y_k) (W_{k'+1} W_{k''} - W_{k'} W_{k''+1}) - g_{k'}(Y_k) (W_{k+1} W_{k''} - W_k W_{k''+1}) + g_{k''}(Y_k) (W_{k+1} W_{k'} - W_k W_{k'+1})
    \]
    for $0 \leq k < k' < k'' \leq d$. Then the natural map $\phi:U'_{\leq -1} \rightarrow (U/J)_{\leq -1} $ is surjective and has kernel $J'$.
\end{proposition}

\begin{proof}
    The proof is identical to that of Proposition~\ref{prop:universalWideal}, with $\bar U$ being defined as an algebra over $\bar S = S/(g_d(Y_j))$ for the induction step in $d$.
\end{proof}

\end{document}